\documentclass[11pt,letterpaper]{amsart}
\usepackage{fullpage,url,amssymb}
\usepackage[pdfencoding=auto,psdextra, colorlinks=true, linkcolor = blue, citecolor=blue, bookmarksdepth=10]{hyperref}
\usepackage{bookmark}
\usepackage[all]{mathtools}
\usepackage{mathrsfs}
\usepackage{mathdots}
\usepackage{amsxtra}
\usepackage{amsrefs}
\usepackage{enumerate, enumitem}
\usepackage[all,arc]{xy}

\newcommand{\Z}{\mathbb Z}

\newcommand{\QQ}{\mathfrak Q}
\newcommand{\K}{\mathfrak K}

\newcommand{\p}{\mathfrak p}

\renewcommand{\o}{\mathfrak o}
\renewcommand{\k}{\mathfrak k}

\renewcommand{\a}{\mathfrak a}
\renewcommand{\b}{\mathfrak b}
\renewcommand{\c}{\mathfrak c}

\theoremstyle{plain}
\newtheorem{definition}{Definition}[section]
\newtheorem{theorem}{Theorem}[section]
\newtheorem{corollary}[theorem]{Corollary}

\newtheorem{lemma}[theorem]{Lemma}
\newtheorem{proposition}[theorem]{Proposition}
\numberwithin{equation}{section}
\theoremstyle{remark}
\newtheorem{remark}{Remark}

\newtheorem*{acknowledgments}{Acknowledgments}

\begin{document}
\title{The Langlands-Shahidi method for pairs via types and covers}
\author{Yeongseong Jo and M. Krishnamurthy}
\address{Department of Mathematics, The University of Iowa, Iowa City, IA 52242}
\curraddr{Department of Mathematics and Statistics, The University of Maine, Orono, ME 04469}
\email{yeongseong.jo@maine.edu}
\address{Department of Mathematics, The University of Iowa, Iowa City, IA 52242}
\email{muthu-krishnamurthy@uiowa.edu}


\begin{abstract}
We compute the local coefficient attached to a pair $(\pi_1,\pi_2)$ of supercuspidal (complex) representations of the general linear group using the theory of types and covers \`{a} la Bushnell-Kutzko. In the process, we obtain another proof of a well-known formula of Shahidi for the corresponding Plancherel constant. The approach taken here can be adapted to other situations of arithmetic interest within the context of the Langlands-Shahidi method, particularly, to that of a Siegel Levi subgroup inside a classical group. 
\end{abstract}
\maketitle

\section{Introduction}
\label{Intro}
Throughout this paper $F$ will denote a non-archimedean local field with residue field cardinality $q$. We fix an additive character $\psi$ which is trivial on $\p_F$ (the maximal ideal of the ring of integers $\o_F$ of $F$) but non-trivial on $\o_F$. In \cite{KK}, the second author with Phil Kutzko outlined a method for calculating the Langlands-Shahidi local coefficient using types and covers via the example of $SL_2(F)$. In this paper, we extend that approach to compute the local coefficient $C_{\psi}(s,\pi_1\times\pi_2)$ attached to a pair $(\pi_1,\pi_2)$ of supercuspidal (complex) representations of the general linear group $GL_n(F)$ and a complex parameter $s$. This complements the work of Paskunas and Stevens \cite{SP} in that we implement a parallel calculation in the context of the Langland-Shahidi method. However, the methods employed here, particularly in the second half of the calculation, are disjoint from that of loc.cit. due to complications arising from possible poles of a certain intertwining operator.

In general, local coefficient by definition is a constant of proportionality arising from uniqueness of induced Whittaker models. Shahidi defined the so-called Langlands-Shahidi (LS) $\gamma$-factors inductively (cf. \cite[Theorem  3.5]{shahidi-plancherel}) so that the local coefficient factorizes as a product of such $\gamma$-factors. The local coefficient is also related to the Plancherel constant (see loc.cit.) which more or less says that ``the square of the local coefficient equals the associated Plancherel constant". We recall the precise relation to the Plancherel constant $\mu(s,\pi_1\times\pi_2)$ in our situation in Subsection~\ref{revisit}. From this standpoint, what we achieve here vis-\`{a}-vis the work of Bushnell, Henniart and Kutzko \cite{BHK} is determine a ``sign" of the square root of the Plancherel constant using types and covers. 
\par

On the other hand, Jacquet, Piatetski-Shapiro, and Shalika \cite{JPSS3} defined the Rankin-Selberg $\gamma$-factor $\gamma(s,\pi_1 \times \pi_2,\psi)$ via the theory of integral representations. By definition it is a proportionality factor between two integrals related to each other by the theory of Fourier transforms. It is known that $\gamma(s,\pi_1\times\pi_2,\psi)$ is a rational function in $q^{-s}$ and can be written in the form 
\[
\gamma(s,\pi_1\times\pi_2,\psi)=\varepsilon(s,\pi_1\times\pi_2,\psi)\frac{L(1-s,\check{\pi}_1\times\check{\pi}_2)}{L(s,\pi_1\times\pi_2)},
\]
where $L(s,\pi_1\times\pi_2)$ is the Rankin-Selberg local $L$-function, $\check{\pi}_i, i=1,2$, denotes the contragredient representation, and $\varepsilon(s,\pi_1\times\pi_2,\psi)$ is the local $\varepsilon$-factor -- a monomial of the form 
\[
 \varepsilon(s,\pi_1 \times \pi_2,\psi)=\varepsilon(0,\pi_1 \times \pi_2,\psi)q^{-f(\pi_1 \times \pi_2,\psi)s}.
\]
The exponent $f(\pi_1 \times \pi_2,\psi)\in {\mathbb Z}$ is called the {\it conductor} attached to the pair $(\pi_1,\pi_2)$; its relation to $\psi$ is given by the equation $f(\pi_1\times\pi_2,\psi)=f(\pi_1\times\pi_2)-n^2{\ell}_{\psi}$, where $\ell_{\psi}$ is the level of the additive character $\psi$. By choice, $\ell_{\psi}=1$ for us. 

In an influential paper \cite{Sha6}, Shahidi proved the equality
\begin{equation}\label{eqn-equal}
C_{\psi}(s,\pi_1\times\pi_2)=\omega_{\pi_2}(-1)^n\gamma(s,\pi_1\times\pi_2,\psi)
\end{equation}
 after suitably normalizing the measures defining the local coefficient. As a consequence, he obtained a formula for $\mu(s,\pi_1\times\pi_2)$: If $\pi_1$,$\pi_2$, are unitary, then 
 \begin{equation}\label{eqn-plan}
 \mu(s,\pi_1\times\pi_2)=q^{f(\pi_1\times\check{\pi}_2,\psi)}\frac{L(s,\pi_1\times\check{\pi}_2)}{L(1+s,\pi_1\times\check{\pi}_2)}\frac{L(-s,\check{\pi}_2\times\pi_1)}{L(1-s,\check{\pi}_2\times\pi_1)}.
 \end{equation}

Our goal in this paper is to provide an alternative (algebraic) approach to obtaining (\ref{eqn-equal}) and (\ref{eqn-plan}) using the theory of types and covers. We believe it opens the door to proving similar equalities in other contexts and consider the present paper a first step in that direction. A case in point is comparison of exterior square local factors obtained from the Langlands-Shahidi method with those obtained via Bump-Friedberg integrals \cite{BuFe}. The authors plan to investigate this in future papers. 

As observed in \cite{KK}, the first instance of calculating the local coefficient along the lines proposed here goes back to Casselman and Shalika \cite{cass1} who computed local coefficients attached to unramified principal series representations using the trivial representation of the Iwahori subgroup -- a special instance of a ``type". In this paper, we use the full force of the theory of types \`{a} la Bushnell and Kutzko \cite{Kut3} to compute $C_{\psi}(s,\pi_1\times\pi_2)$. Let us now give a brief overview of the methods used in this paper and also comment on the organization of its contents. 

We take $\pi_1$ and $\pi_2$ to be supercuspidal representations that contain the same simple character. By \cite{Kut3} (also see \cite[Section 7]{Kut2}), for $i=1,2$, we may choose maximal simple types $(J_i,\lambda_i)$ contained in $\pi_i$ so that $J_1=J_2=J(\beta,\a)$ is a compact open subgroup associated to a maximal simple stratum $[\a,k,0,\beta]$, the representation $\lambda_i$ has a decomposition of the form $\lambda_i=\kappa\otimes\tau_i$, where $\kappa$ is a $\beta$-extension and $\tau_i$ is the inflation of an irreducible cuspidal representation of $J(\beta,\a)/J^{1}(\beta,\a)$. As observed in \cite{bush2, SP}, it is a formal consequence of Mackey's theorem that we may also arrange the choice of these maximal simple types so that ${\rm Hom}_{U_n(F)\cap J(\beta,\a)}(\psi,\lambda_i)\neq 0$. Without loss of any generality, we may take $\pi_1,\pi_2$, to be unitary.
Assuming all this is done, Paskunas and Stevens \cite{SP} defined a pair of distinguished Whittaker functions $({\mathcal W}_1,{\mathcal W}_2)$ in the Whittaker model (w.r.t. $\psi$) of $\pi_1,\pi_2$, respectively, that has many useful properties. We give the necessary definitions and review these properties in Subsections \ref{T} and \ref{whitt}. We claim no new results here, but the reader may find our exposition pertaining to these Whittaker functions useful.  

Now, consider the maximal Levi subgroup $L=GL_n(F)\times GL_n(F)$ inside the group $G=GL_{2n}(F)$ and let $P=LN$ be the associated standard parabolic subgroup. Then $(J_L,\lambda_L)$, $J_L=J(\beta,\a)\times J(\beta,\a)$, $\lambda_L=\lambda_1\times\lambda_2$, is a type in $L$ associated to the $L$-inertial class of $\pi_1\times\pi_2$. Let $(J',\lambda')$ be the corresponding $G$-cover (we discuss this in Section \ref{cov}) as constructed in \cite[Section 7]{Kut2}. This comes equipped with an injective algebra homomorphism ${\mathcal H}(L,\lambda_L)\xrightarrow{j_P} {\mathcal H}(G,\lambda')$ of associated Hecke algebras that realizes parabolic induction. Then one can pass between the category of smooth representations and the corresponding modules over these algebras which plays a crucial role in our computation. In Section \ref{sec-RS}, we review the work by Paskunas and Stevens and start our calculation of the local coefficient $C_{\psi}(s,\pi_1\times\pi_2)$ in Section \ref{LS}. We break into two cases: (i) $\tau_1\ncong \tau_2$ and (ii) $\tau_1\cong \tau_2$. 

We deal with case (i) in Subsection~\ref{sec:unequal}. In this situation the cover $(J',\lambda')$ splits, meaning, the above map $j_P$ is an isomorphism. This makes it is easy to determine the effect of the intertwining operator in question. After suitably conjugating the cover $(J',\lambda')$ by a central element in $L$, we compute both sides of the equation (i.e., \eqref{lc}) defining the local coefficient using the pair $({\mathcal W}_1,{\mathcal W}_2)$. (This is similar to the approach in \cite{SP}.) We partition the relevant integral into ``shells" and prove that up to certain precise volume factors associated with the cover the local coefficient is given as 
\[
C_{\psi}(s,\pi_1\times\pi_2)\sim \int\limits_{J(\beta,\a)}{\mathcal W}_1(\varpi_E^mX){\mathcal W}_2(\varpi_E^mX)\phi_m(X)dX.
\]
Here $m$ is so-called numerical invariant (see Subsection \ref{sec-ni}) which is closely related to the conductor $f(\pi_1\times\check{\pi}_2,\psi)$ and $\phi_m(X)$ is the additive character given by $X\mapsto \psi(\varpi^m_EX_{n,1})$. These put together is the content of Theorem \ref{LS-equal} from which we can deduce \eqref{eqn-equal} in the case at hand. It is likely that the above integral can be expressed as a ``generalized Gauss sum" using properties of $({\mathcal W}_1,{\mathcal W}_2)$ but we have not pursued it here. (See \cite{YeZe} for a related discussion.) In any case, we determine the absolute value of this integral in Subsection \ref{revisit} using the local functional equation. 

We treat case (ii) in Subsection~\ref{sec:equal}. Here, the cover $(J',\lambda')$ is not a split cover and the intertwining operator is not well-behaved. So we cannot proceed as before, instead we use the Hecke algebra isomorphism of \cite{Kut3}. It is proved in loc.cit. that the Hecke algebra ${\mathcal H}(G,\lambda')$ is isomorphic to the Iwahori Hecke algebra of $G'=GL_2(\k)$ for a suitable field extension $\k$ of $F$. We then use the ``generalized spherical vector" defined in \cite{Kim} and transport the corresponding Whittaker function across this Hecke algebra isomorphism using results of Chan and Savin \cite{ChanSavin1, ChanSavin2}. (See Proposition~\ref{GL(n)-Whittaker} for a precise statement.) This reduces the problem of calculating the local coefficient to the aforementioned computation of Casselman and Shalika. We give the final expression for the local coefficient in the non-split case in Theorem~\ref{thm-nonsplit-lc}. Our proof involves a careful analysis of the Hecke algebra isomorphism, in particular, we resolve the sign ambiguities mentioned in \cite[Remark 4.2.6]{Kim}. To conclude, in Subsection~\ref{revisit}, we deduce \eqref{eqn-plan} using certain volume computations. 

We expect the simplifying assumption that $\pi_1$ and $\pi_2$ belong to the same endo-class is not necessary. (See \cite{Kim14} for progress in this direction.) 

\section{Types and Whittaker Functions}
\subsection{Maximal simple types}
\label{T}
In this section we review the structure of irreducible supercuspidal representations of $G=GL_n(F)$ via Bushnell-Kutzko's theory of types. The definitive reference for the theory is \cite{Kut3} and we adopt the notation there with minor modifications. 
Let $A$ denote the algebra $M_{n}(F)$ of $F$-endomorphisms of $F^n$ and let $\a$ be a {{\it hereditary}} $\o_F$-order in $A$. Let $\p=\p_{\a}$ denote its Jacobson radical, a two-sided ideal of $\a$. Let $U(\a)$ denote the group of units ${\a}^{\times}$ and set $U^k(\a)=1+\p^k$ for $k\geq 1$. These are compact open subgroups of $G$. Let ${K}(\a)=N_{G}(U(\a))$ (or equivalently, defined as the $G$-normalizer of $\a$), then $K(\a)$ is a open compact-modulo-center subgroup of $G$. It is useful to know that the $K(\a)$ also normalizes the subgroups $U^k(\a), k\geq 1$. Let $v_{\a}$ be the valuation map associated with the hereditary order $\a$. This induces a surjective homomorphism from 
\[
K(\a)\longrightarrow \Z.
\]

We say $\a$ is a {\it principal hereditary order} if the ideal $\p_{\a}$ is principal. In this situation, the group $K(\a)$ is a maximal compact mod center subgroup of $G$ and all such subgroups of $G$ arise this way. Further,  for $\pi_{\a}$ satisfying $v_{\a}(\pi_{\a})=1$, we have 
\[
\p=\a\pi_{\a}=\pi_{\a}\a\mbox{ and }K(\a)=\langle\pi_{\a}\rangle U(\a).
\]
Any principal hereditary order $\a$ is $G$-conjugate to the the order of matrices over $\o_F$ which are upper triangular (in blocks) modulo $\p_F$, where each block is of size $n/e$ and the number of blocks $e=e_\a$ is the period of the lattice chain ${\mathcal L}$ associated to $\a$. 

Let  $[\a,k,0,\beta]$, $k\geq 1$, be a {\it principal simple stratum} in $A$ (see \cite[(1.5.5)]{Kut3}). It consists of a principal hereditary $\o_F$-order $\a$ in $A$ and a matrix $\beta\in A$ satisfying 
\begin{enumerate}
\item[(i)] the algebra $E=F[\beta]$ is a field, whose degree over $F$ is denoted $d$,
\item[(ii)] $E^{\times}\subset K(\a)$,
\item[(iii)] $v_{\a}(\beta)=-k$,
\end{enumerate}
and another technical condition denoted as (iv) in loc.cit. Let $B$ denote the $A$-centralizer of $\beta$ and put $\b=\a\cap B$, $\QQ=\text{rad}(\b)$. The $E$-algebra $B$ is isomorphic to $M_{n/d}(E)$ and $\b$ is a hereditary $\o_E$-order in $B$. The stratum is said to be {\it maximal} if $\b$ is a maximal $\o_E$-order in $B$; given an isomorphism of $E$-algebras $B\cong M_{n/d}(E)$, one identifies $\b$ with the standard maximal order $M_{n/d}(\o_E)$. The lattice period $e_{\a}$ in this case is same as the  ramification index $e(E/F)$ of $E/F$.  Attached to such a stratum are a pair of $\o_F$-orders given by ${\mathfrak H}(\beta,\a)\subseteq {\mathfrak J}(\beta,\a)\subseteq \a$ given by 
\[
{\mathfrak H}(\beta,\a)=\b+\p^{[\frac{n}{2}]+1}; {\mathfrak J}(\beta,\a)=\b+\p^{[\frac{n+1}{2}]}
\]
which gives the compact open subgroups $H(\beta,\a)={\mathfrak H}(\beta,\a)^{\times}$ and $J(\beta,\a)={\mathfrak J}(\beta,\a)^{\times}$ of $G$. These are filtered by $H^m(\beta,\a)=H(\beta,\a)\cap U^m(\a)$, $J^m(\beta,\a)=J(\beta,\a)\cap U^m(\a)$,  $m\geq 0$, where $U^{0}(\a)=\a^{\times}$. 
In particular, we have compact open subgroups
\[
H^1(\beta,\a)\subseteq J^{1}(\beta,\a)\subseteq J(\beta,\a) 
\]
of $U(\a)$. There is a finite set of characters ${\mathcal C}(\a,\beta)$ of $H^{1}(\beta,\a)$ called {\it simple characters}. By construction, $J(\beta,\a)$ normalizes $H^{1}(\beta,\a)$ and ${\mathcal C}(\a,\beta)$ depends on the choice of an additive character $\psi=\psi_F$ of $F$ of level one which we fix throughout this paper. 

For a maximal (principal) simple stratum $[\a,k,0,\beta]$ as above, put 
\[
\tilde{J}(\beta,\a)=E^{\times} J(\beta,\a)
\]
which is a compact mod center subgroup of $G$. The data comprising these subgroups and the set ${\mathcal C}(\a,\beta)$ of simple characters are at the core of the classification of supercuspidal representations of $G$. For a summary of their properties, see \cite[(2.1.1)]{B-Hen}. Here, we highlight that 
\begin{enumerate}[label=$(\mathrm{\alph*})$]
\item\label{B-Hen1} $J(\beta,\a)$ is the unique maximal compact subgroup of $\tilde{J}(\beta,\a)$. 
\item\label{B-Hen2} $J(\beta,\a)=U(\b)\cdot J^{1}(\beta,\a)$ with $U(\b)\cap J^{1}(\beta,\a)=U^{1}(\b)$.
\item\label{B-Hen3} The normalizer of any simple character $\theta\in {\mathcal C}(\a,\beta)$ in $G$ is $\tilde{J}(\beta,\a)$.
\item\label{B-Hen4} Given a $\theta\in {\mathcal C}(\a,\beta)$, there is a unique irreducible representation $\eta$ of $J^{1}(\beta,\a)$ containing $\theta$.
\end{enumerate}

For $k=0$, we set $E=F$ and take $\a$ to be a maximal $\o_F$-order and deem the resulting $[\a,0,0,0]$ a maximal simple stratum as well. In this situation $J(0,\a)=U(\a)$, is a maximal compact open subgroup of $G$, $J^{1}(0,\a)=H^{1}(0,\a)=U^{1}(\a)$, and $\tilde{J}(0,\a)=K(\a)=F^{\times}U(\a)$. By a simple character in this situation, we mean the {\it trivial character} of $U^{1}(\a)$. 

\begin{definition}
A pair $(J,\lambda)$, where $J$ is a compact open subgroup of $G$ and $\lambda$ is an irreducible representation of $J$,  is said to be a {\it maximal simple type} if there is a maximal principal simple stratum $[\a,k,0,\beta]$ (including the case $k=0$ in the above sense) and a simple character $\theta\in {\mathcal C}(\a,\beta)$, satisfying $J=J(\beta,\a)$ and $\theta$ is contained in the restriction of $\lambda$ to $H^{1}(\beta,\a)$. 
\end{definition}
The simple character $\theta$ is said to be {\it attached} to $\lambda$. Let ${\mathcal A}_{n}^{0}(F)$ denote the category of irreducible admissible supercuspidal (complex) representations of $G$. 
One of the main results in \cite{Kut3} (see Ch. 6) on the classification of supercuspidal representations, in terms of maximal simple types, is the following:
\begin{proposition}\label{prop-BK}
Suppose $\sigma\in {\mathcal A}_{n}^{0}(F)$.
\begin{enumerate}[label=$(\mathrm{\alph*})$]
\item There exists a {\it{maximal simple type}} $(J,\lambda)$ which is uniquely determined up to $G$-conjugacy, so that that the restriction of $\sigma$ to $J$ contains $\lambda$. 
\item Let $[\a,k,0,\beta]$ be a maximal simple stratum such that $J=J(\beta,\a)$ and $\theta\in {\mathcal C}(\a,\beta)$. Then $\lambda$ extends uniquely to a representation $\tilde{\lambda}$ of the normalizer $\tilde{J}:=\tilde{J}(\beta,\a)$ of $\theta$ in $G$ such that the compact induction of $\tilde{\lambda}$ is isomorphic to $\sigma$. 
\end{enumerate}
\end{proposition}
A pair $(\tilde{J},\tilde{\lambda})$ arising in this manner is called an {\it extended maximal simple type}. If $\theta$ is the simple character attached to $\lambda$, the following also holds: 
\[
\lambda\cong \kappa\otimes \tau,
\]
where $\kappa$ is a $\beta$-{\it extension} of the unique irreducible representation $\eta$ of $J^1=J^{1}(\beta,\a)$ containing $\theta$, and $\tau$ is the inflation to $J$ of  a cuspidal representation of $J(\beta,\a)/J^1(\beta,\a)\cong U(\mathfrak b)/U^{1}(\mathfrak b)\cong GL_{n/d}(k_E)$. If $k=0$, the representation $\sigma$ is said to be of ``level zero''. In this case, $\theta$, $\eta$ and $\kappa$ are all trivial, and $\lambda$ is the inflation of an irreducible cuspidal representation of $GL_n(k_F)$. Otherwise, $\sigma$ is of ``positive level'' and $k$ is the smallest integer so that $U^{k+1}(\mathfrak a)\subset \text{ker }\sigma$. The {\it level} $l_{\sigma}$ (normalized) of $\sigma$ is defined as ${k}/{e}$, where $e=e_{\a}=e(E/F)$ with $E=F[\beta]$. 

In both situations, the group $\tilde{J}\subseteq K(\a)$. If we put $\tilde{\rho}=\mbox{Ind}_{\tilde{J}}^{K(\mathfrak a)}\tilde{\lambda}$, then by transitivity of induction $\sigma\cong \mbox{c-Ind}_{K(\a)}^{G}\tilde{\rho}$, consequently $\tilde{\rho}$ is irreducible. One can check using Mackey formula that 
\[
\tilde{\rho}|_{U(\a)}\cong \rho:=\mbox{Ind}_{J(\beta,\a)}^{U(\a)}\lambda
\] 
and that it is also irreducible. 

Let ${\mathfrak s}=[G,\sigma]_{G}$ be the supercuspidal inertial class in $G$ determined by $\sigma$. By \cite[6.2.3]{Kut3}, an irreducible representation $\sigma'$ of $G$ contains $\lambda$ (or equivalently $\rho$)  if and only if $\sigma'\simeq \sigma\otimes\chi$ for some unramified quasicharacter $\chi$ of $F^{\times}$. This is what one means by ``$(J(\beta,\a),\lambda)$ (or $(U(\a),\rho)$) is a ${\mathfrak s}$-type". 
\subsection{Explicit Whittaker Functions}
\label{whitt}
In this section we continue with $\sigma\in {\mathcal A}_{n}^{0}(F)$ and review the construction of a certain special Whitakker function in the Whittaker model of $\sigma$ due to Paskunas and Stevens \cite{SP}. Let ${\mathfrak K}$ be an open, compact mod centre subgroup of $G$ and let $(\Lambda,W)$ be a smooth irreducible representation of ${\mathfrak K}$ such that $\sigma\cong \mbox{c-Ind}_{\mathfrak K}^{G}\,\Lambda$. Let $U$ be the unipotent subgroup of $G$ consisting of upper triangular unipotent matrices and let $B=TU$ be the corresponding Borel subgroup of $G$. 

Let ${\mathcal C}$ be the class of functions in $\mbox{c-Ind}_{\mathfrak K}^{G}\,\Lambda$ that are supported in ${\mathfrak K}$. Then there is a canonical ${\mathfrak K}$-embedding, $w\mapsto \varphi_w$, from $W$ onto ${\mathcal C}$ given by 
\[
\varphi_w(g)=\left\{\begin{array}{ll}\Lambda(g)w,&\mbox{if }g\in{\mathfrak K};\\0,&\mbox{if }g\not\in\mathfrak{K}.\end{array}\right.
\]

Fix a smooth character $\psi=\psi_F$ of $F$ of level one (as before), trivial on $\p_F$, but not on $\o_F$. Then this determines a smooth non-degenerate character of $U$, also denoted $\psi$ by abuse of notation, via
\begin{equation}\label{psi}
\psi(u)=\psi(\sum_{i=1}^{n-1} u_{i,i+1}), u=(u_{ij})\in U.
\end{equation}
It is well-known that $\sigma$ is generic and that $\mbox{dim}_{\mathbb{C}}\mbox{ Hom}_{G}(\mbox{c-Ind}_{\mathfrak K}^{G}\,\Lambda,\mbox{Ind}_{U}^{G}\,\psi)=1$. This space may be described using Mackey theory \cite{Kut4}. Namely, let $\mathcal{H}(G,\Lambda,\psi)$ denote the space of functions, 
\[
f: G\longrightarrow \mbox{Hom}_{\mathbb{C}}(W,\mathbb{C}),
\]
which satisfy 
\[
f(ugk)=\psi(u)f(g)\circ\Lambda(k), u\in U, g\in G, k\in {\mathfrak K}.
\]
Let $dx$ denote the Haar measure on $G/F^{\times}$. Then, for $\phi\in\mbox{c-Ind}_{\mathfrak K}^{G}\,\Lambda$, $f\in {\mathcal H}(G,\Lambda,\psi)$, we can form the convolution 
\begin{equation}
\label{conv}
f\star\phi(g)=\int\limits_{G/F^{\times}}f(y)(\phi(y^{-1}g))dy, g\in G.
\end{equation}
One checks that the function $f\star \phi$ belongs to $\mbox{Ind}_{U}^{G}\,\psi$ and thus this determines a $G$-homomorphism from 
\begin{equation}
\label{eqn-W}
\mathcal{H}(G,\Lambda,\psi)\longrightarrow \mbox{ Hom}_{G}(\mbox{c-Ind}_{\mathfrak K}^{G}\,\Lambda,\mbox{Ind}_{U}^{G}\,\psi)
\end{equation}
which according to \cite{Kut4} is an isomorphism. Since the right hand side of (\ref{eqn-W}) is one dimensional and $\sigma$ is irreducible, there is a unique $G$-subspace ${\mathcal W}(\sigma,\psi)$ of $\mbox{Ind}_{U}^{G}(\psi)$ which is isomorphic to $\sigma$; we call this the $\psi$-Whittaker model of $\sigma$. We write 
\[
\sigma\cong \mbox{c-Ind}_{\mathfrak K}^{G}\,\Lambda\ni \phi\mapsto {\mathcal W}_{\phi}\in \mathcal{W}(\sigma,\psi)
\]
to denote this bijection. 

 It follows from (\ref{eqn-W}) that there is a unique double coset $Ux{\mathfrak K}$ that supports a nonzero element of $\mathcal{H}(G,\Lambda,\psi)$ and the space of such functions is one dimensional. This in turn means that there is a unique $x$ such that $\Lambda$ contains the character $\psi^{x}$ of $x^{-1}Ux\cap {\mathfrak K}$ with multiplicity $1$, or equivalently, the dual representation $(\check{\Lambda},\check{W})$ contains the inverse of the character $\psi^{x}$ with multiplicity $1$. Let $\mu$ be a nonzero element of $\check{W}=\mbox{Hom}_{\mathbb{C}}(W,\mathbb{C})$ that transforms according to $(\psi^{x})^{-1}$ when restricted to $x^{-1}Ux\cap {\mathfrak K}$. Define $f=f_{\mu}\in{\mathcal H}(G,\Lambda,\psi)$ supported on $Ux{\mathfrak K}$ as 
\[
f(uxk)=\psi(u)(\mu\circ\Lambda(k)), u\in U, k\in{\mathfrak K}.
\]
It is a simple matter to check that $f$ is well-defined. The aforementioned bijection is then given by 
\begin{equation}\label{bij}
\phi\mapsto {\mathcal W}_{\phi}: =f\star \phi,
\end{equation}
and the corresponding $\psi$-Whittaker functional $\Omega (=\Omega^f)$ is given by $\Omega(\phi)=(f\star\phi)(1)$. In particular, $\Omega(\varphi_w)=0$ unless $x$ represents the trivial double coset $U{\mathfrak K}$. Since $f$ is unique up to a scalar, we suppress the obvious dependence of the function ${\mathcal W}_{\phi}$ and the functional $\Omega$ on $f$. In any case, the Whittaker space ${\mathcal W}(\sigma,\psi)$ is independent of the choice of $f$.  

For calculation purposes, it will be convenient to have the coset containing the identity as the one that supports the Whittaker model. To this end, keeping the above notation, conjugating by $x$ we see that 
\[
\sigma\cong \mbox{c-Ind}_{x{\mathfrak K}x^{-1}}^{G}\,\Lambda^{x^{-1}}.
\]
This has the effect of changing ${\mathfrak K}\mapsto x{\mathfrak K}x^{-1}$, $\Lambda\mapsto \Lambda^{x^{-1}}$ and $f\mapsto R(x)f$. Then $\Lambda^{x^{-1}}$ contains $\psi$ when restricted to $x{\mathfrak K}x^{-1}\cap U$ and $R(x)f$ is supported in $Ux{\mathfrak K}x^{-1}$. So without loss of any generality, we may assume $\Lambda$ contains $\psi$ while writing $\sigma$ as a compactly induced representation from a compact mod center subgroup of $G$. 

\begin{proposition}\label{Whitt}
Suppose $\sigma\cong \mbox{c-}{\rm Ind}_{\mathfrak K}^{G}\Lambda$ is as above and assume $\mathrm{Hom}_{U\cap \K}(\psi,\Lambda)\neq 0$. Then there exists a Whittaker function ${\mathcal W}\in {\mathcal W}(\sigma,\psi)$ whose support is contained in $U\K$ and satisfying the following properties: 
\begin{itemize}
\item ${\mathcal W}(1)=1$
\item ${\mathcal W}(gu)=\psi(u){\mathcal W}(g), u\in U\cap \K$.
\end{itemize}
Further, if $\sigma$ is unitarizable, we may choose $W$ so that it also satisfies  
\[
{\mathcal W}(g^{-1})=\overline{{\mathcal W}(g)}, g\in \K.
\] 
\end{proposition}
\begin{proof}
Since $\Lambda$ contains $\psi$, we may take $x=1$ in the above discussion. Choose $0\neq \mu\in \check{W}$ that transforms according to $\psi^{-1}$ on $U\cap\K$. The space of such functionals is one dimensional. Let $f\in {\mathcal H}(G,\Lambda,\psi)$ be the function supported in $U\K$ defined by $f(1)=\mu$. Then, for the measure $du$ normalized so that $\text{vol}(U\cap \K)=1$, and after rescaling $\mu$, one checks that the bijection $\phi\mapsto {\mathcal W}_{\phi}$, $\phi\in \mbox{c-Ind}_{\K}^{G}\Lambda$, is given by 
\[
{\mathcal W}_{\phi}(g)=(f\star\phi)(g)=\int\limits_{U}\psi^{-1}(u)\mu(\phi(ug))du.
\]
Now, since the dimension of the space $\mbox{Hom}_{U\cap \K}(\psi,\Lambda)$ is one, there is a unique $w\in W$ that transforms according to $\psi$ and satisfying $\mu(w)=1$. Let $c_{\mu,w}(g)=\mu({\Lambda}(g)w), g\in  \K$, denote the matrix coefficient associated to the pair $(\mu,w)$. Put ${\mathcal W}={\mathcal W}_{\phi_w}$. Since $\phi_w$ is supported in $\K$, it follows that ${\mathcal W}$ is supported in $U\K$. It also follows from the above formula that, for $k\in \K$, 
\[
{\mathcal W}(k)=\int\limits_{U\cap \K}\psi^{-1}(u)\mu(\phi_w(uk))du=\int\limits_{U\cap\K}\psi^{-1}(u)\mu(\Lambda(uk)w)du=\mu(\Lambda(k)w).\nonumber
\]
Here, the last equality follows since by choice $\mu$ has the property 
\[
\mu(\Lambda(u)w')=\psi(u)\mu(w'), u\in U\cap \K, w'\in W.
\]
Thus ${\mathcal W}(uk)=\psi(u)c_{\mu,w}(k), u\in U, k\in \K$ and it clearly has the required properties. Regarding the final assertion, if $\sigma$ is unitarizable, so is the representation $\Lambda$. Fix a $\K$-invariant Hermitian inner product $\langle\cdot,\cdot\rangle$ on $W$ and identify $\check{W}$ with $W$. Choose $w \in W$ so that it transforms according to $\psi$ on $U\cap \K$ and satisfies $\langle w,w\rangle=1$. Now apply the above argument with $\mu=\mu_w$, where $\mu_w(w')=\langle w',w\rangle, w'\in W$, to see that ${\mathcal W}(k)=\langle \Lambda(k)w,w \rangle$, $k\in \K$. 
\end{proof}
By Proposition~\ref{prop-BK}, $\sigma$ determines a maximal simple type $(J,\lambda)$ that is unique up to $G$-conjugacy, we may therefore after conjugation (if necessary) choose $\K=\tilde{J}$ and $\Lambda=\tilde{\lambda}$, so that $\tilde{\lambda}$ contains the character $\psi$ when restricted to $U\cap \tilde{J}$. This gives us the following (cf. \cite[Proposition 1.6]{bush2}, \cite[Proposition 1.3]{SP}):
\begin{proposition}\label{l1}
Suppose $\sigma\in {\mathcal A}_n^{0}(F)$. There exists an extended maximal simple type $(\tilde{J},\tilde{\lambda})$ with associated principal stratum $[\a,k,0,\beta]$ , $k\geq 0$, as in \S\ref{T}, satisfying 
\[
\sigma\cong \mbox{c-}\mathrm{Ind}_{\tilde{J}}^{G}\tilde{\lambda} \quad \text{and} \quad \mathrm{Hom}_{U\cap\tilde{J}}(\psi,\tilde{\lambda})\neq 0.
\]
The conclusion also remains valid if we replace $(\tilde{J},\tilde{\lambda})$ with the associated $(K(\a),\tilde{\rho})$. Moreover, such a pair $(\tilde{J},\tilde{\lambda})$ is determined up to conjugation by $u\in U$. 
 \end{proposition}

Fix an extended maximal simple type $(\tilde{J},\tilde{\lambda})$ given by the above proposition and 
let $\theta\in {\mathcal C}(\a,\beta)$ be the simple character of $H^1=H^1(\beta,\a)$ attached to $\tilde{\lambda}$.  The fact that $\tilde{\lambda}\supset \psi$ implies that 
\[
\psi(x)=\theta(x), x\in U\cap H^{1}.
\]
Define the character $\Psi: (J(\beta,\a)\cap U)H^{1}\longrightarrow \mathbb{C}^{\times}$ as in \cite[Definition 4.2]{SP} via 
\[
\Psi(uh):=\psi(u)\theta(h)
\]
which is well-defined since $J$ normalizes $\theta$. The character $\Psi$ occurs in $\tilde{\lambda}$ (and also in $\sigma$) with multiplicity one. Applying Proposition~\ref{Whitt} to the pair $(\K,\Lambda)=(\tilde{J},\tilde{\lambda})$, let ${\mathcal W}_{\sigma}\in {\mathcal W}(\sigma,\psi)$ denote the resulting Whittaker function. 

The main thrust of \cite[Section 5]{SP} is that ${\mathcal W}_{\sigma}$ can be realized in terms of the {\it Bessel function} which reveals additional properties of ${\mathcal W}_{\sigma}$ that are crucial for computation of the Rankin-Selberg local factors. Namely, put ${\mathcal U}=(U\cap J(\beta,\a))H^{1}$, ${\mathcal M}=(P^{1}\cap J(\beta,\a))J^{1}$, and ${\mathcal K}=\tilde{J}$, where $P^{1}$ is the mirabolic subgroup of $G$ consisting of matrices whose last row is $e_n:=(0,0,\ldots,0,1)$ and $J^{1}=J^{1}(\beta,\a)$. Thus we have the data ${\mathcal U}\subset {\mathcal M}\subset {\mathcal K}$ along with the representation $\tilde{\lambda}$ of ${\mathcal K}$ and the character $\Psi$ of ${\mathcal U}$ satisfying  \cite[Theorem 4.4]{SP}:
\begin{itemize}
\item $\tilde{\lambda}|_{\mathcal M}$ is irreducible,
\item $\tilde{\lambda}|_{\mathcal M}\cong \mbox{Ind}_{\mathcal U}^{\mathcal M}\Psi$.
\end{itemize} 

Attached to this data is the Bessel function ${\mathcal J}={\mathcal J}_{\tilde{\lambda}}: {\mathcal K}\longrightarrow \mathbb{C}$ having the following properties (cf. \cite[Proposition 5.3]{SP}):
\begin{enumerate}[label=$(\mathrm{\roman*})$]
\item\label{Whittaker1} ${\mathcal J}(1)=1$;
\item\label{Whittaker2} ${\mathcal J}(hg)={\mathcal J}(gh)=\Psi(h){\mathcal J}(g)$ for all $h\in {\mathcal U}, g\in {\mathcal K}$;
\item\label{Whittaker3} if ${\mathcal J}(g)\neq 0$, then $g$ intertwines $\Psi$. In particular, for $m\in {\mathcal M}$, ${\mathcal J}(m)\neq 0$ if and only if $m\in {\mathcal U}$; 
\item\label{Whittaker4} for all $g_1,g_2\in {\mathcal K}$, we have 
\[
\sum_{{\mathcal M}/{\mathcal U}}{\mathcal J}(g_1m){\mathcal J}(m^{-1}g_2)={\mathcal J}(g_1g_2).
\]
\end{enumerate}    
By \cite[Proposition 5.7]{SP}, we have 
\[
{\mathcal W}_{\sigma}(g)={\mathcal J}(g), g\in \tilde{J},
\] 
and consequently it follows from \ref{Whittaker3} that, for $g\in P$, 
\[
{\mathcal W}_{\sigma}(g)\neq 0\implies g\in {\mathcal U}.
\] 
\subsection{A numerical invariant}\label{sec-ni}
For $\sigma\in{\mathcal A}^{0}_n(F)$, choose an extended maximal simple type ($\tilde{J}(\beta,\a),\tilde{\lambda})$ with $\tilde{\lambda}\supset \psi$ as in Proposition~\ref{l1}. Recall $A=M_{n\times n}(F)$ and $E=F[\beta]$. Let $\{e_i : 1\leq i\leq n\}$ be the standard row basis of $F^n$. Let $\phi: A\longrightarrow \mathbb{C}^{\times}$ denote the additive character given by 
\[
\phi(X)=\psi(e_nX\prescript{t}{}e_1), X\in A.
\]
Let $m$ be the integer so that  $\phi$ is trivial on $\p^{m+1}$ but not on $\p^m$, we call this the conductor of $\phi$ with respect to the Jacobson radical $\p=\p_{\a}$. If $\nu\in E^{\times}/(1+\p_E)$ denotes the ``numerical invariant'' defined in \cite[Definition 6.1]{SP}, then $m=\dfrac{n\cdot{\rm ord}_E(\nu)}{d}$ according to \cite[Lemma 7.5]{SP}. Also, for any integer $r\in \Z$, let $\phi_r: A\rightarrow \mathbb{C}^{\times}$ denote the additive character $\phi_r(X)=\phi(\varpi_E^{r}X)$; it has conductor $m-r$, i.e., non-trivial on $\p^{m-r}$, but trivial on $\mathfrak{p}^{m-r+1}$. 
\section{Covers}
\label{cov}
\subsection{Some generalities}
\label{gen}
The notion of {\it covers} is a general theory that gives a module theoretic interpretation of {\it{parabolic induction}} and {\it{Jacquet restriction}} in representation theory of $p$-adic groups. We refer the reader to \cite[\S8]{BK} for the foundational aspects of this theory. Suppose $G$ is the group of $F$-points of some connected reductive group defined over $F$ and $P$ is the $F$-points of a parabolic subgroup of $G$. Let $P=LN$ be a Levi decomposition with $L$ the $F$-points of a Levi subgroup of $P$ and $N$ the $F$-points of the unipotent radical of $P$.
Let $X(L)$ denote the group of {\it unramified quasicharacters} of $L$, i.e., continuous homomorphisms $L\longrightarrow \mathbb{C}^{\times}$ that are trivial on all compact subgroups of $L$. 
Let ${\mathfrak R}(G)$ denote the category of smooth complex representations of $G$. 

By a cuspidal pair in $G$ we mean a pair $(L,\tau)$ in $G$, where $L$ is as above and $\tau$ is a supercuspidal representation of $L$. Two such pairs $(L_i,\tau_i), i=1,2,$ are said to be {\it inertially equivalent} if there exists a $g\in G$ and $\chi\in X(L)$ such that 
\[
L_2=L_1^{g}=g^{-1}L_1g\mbox{ and }\sigma_2\cong \sigma_1^{g}\otimes\chi,
\]
where $\sigma_1^g$ is the representation $x\mapsto \sigma_1(gxg^{-1})$ of $L_2$. We write $[L,\tau]$ to denote the $G$-inertial equivalence class of a cuspidal pair $(L,\tau)$ in $G$ and let ${\mathcal B}(G)$ denote the set of inertial equivalence classes in $G$. For each ${\mathfrak s}\in {\mathcal B}(G)$, we have a full subcategory ${\mathfrak R}^{\mathfrak s}(G)$ of ${\mathfrak R}(G)$ defined as follows: a smooth representation $\pi'$ belongs to ${\mathfrak R}^{\mathfrak s}(G)$ if and only if each irreducible subquotient $\pi$ of $\pi'$ has {\it inertial support} ${\mathfrak s}$ (cf. \cite[Definition 1.1]{BK}). 

A pair $(J,\lambda)$, where $J$ is a compact open subgroup of $G$ and $(\lambda,W)$ is a smooth irreducible representation of $J$, is said to be a {\it{${\mathfrak s}$-type}} if the following holds: For every irreducible object $(\pi,V)\in {\mathfrak R}(G)$, $(\pi,V)$ belongs to ${\mathfrak R}^{\mathfrak s}(G)$ if and only if $\pi$ contains $\lambda$, i.e., the space $V_{\lambda}:=\text{Hom}_J(W,V)\neq 0$. Let $(\check{\lambda},\check{W})$ denote the contragredient of $(\lambda,W)$, we define ${\mathcal H}(G,\lambda)$ as the space of compactly supported functions $f:G\longrightarrow \text{End}_{\mathbb{C}}(\check{W})$ that satisfy 
\[
f(hxk)=\check{\lambda}(h)f(x)\check{\lambda}(k), x\in G, h,k\in J.
\]
It is unital (associative) algebra under the standard convolution operation 
\[
f_1\star f_2(g)=\int\limits_{G}f_1(x)f_2(x^{-1}g)dx, \mbox{ with }f_1,f_2 \in {\mathcal H}(G,\lambda).
\] 
One can similarly define the algebra ${\mathcal H}(G,\check{\lambda})$. There is a canonical anti-isomorphism $f\mapsto \check{f}$ from ${\mathcal H}(G,\lambda)\rightarrow {\mathcal H}(G,\check{\lambda})$ given by $\check{f}(g)=(f(g^{-1}))^{\vee}$. For $a\in \mbox{End}_{\mathbb C}(\check{W})$, $a^{\vee}$ denotes the transpose of $a$ with respect to the canonical pairing between $W$ and $\check{W}$. The space $V_{\lambda}$ of $\lambda$-coinvariants then has a natural left ${\mathcal H}(G,\lambda)$-module structure (also denoted as $\pi$) given by 
\[
\pi(f)\phi(w)= \int\limits_{G}\pi(g)\phi(f(g)^{\vee}w)dg; f\in {\mathcal H}(G,\lambda),\phi\in V_{\lambda}, w\in W.
\]
Then the map $V\mapsto V_{\lambda}$ is an equivalence of categories ${\mathfrak R}^{\mathfrak s}(G)\cong {\mathcal H}(G,\lambda)-\text{Mod}$.   

We write $\iota_P^G$ to denote the functor of {\it normalized parabolic induction}. For any smooth representation $\sigma$ of $L$, let ${\mathcal F}_P(\sigma)$ denote the space of ${\iota}_P^G(\sigma)$. 

Let $J_{L}$ be a compact open subgroup of $L$ and $\lambda_{L}$ an irreducible smooth representation of $J_{L}$. A $G$-cover of $(J_{L},\lambda_{L})$ is a pair $(J,\lambda)$, where $J$ is a compact open subgroup of $G$ and $\lambda$ is a smooth irreducible representation of $J$, satisfying certain properties. We refer the reader to \cite[Definition 8.1]{BK} for the precise definition of a $G$-cover. 
Suppose $(J_{L},\lambda_{L})$ is ${\mathfrak t}$-type for ${\mathfrak t}\in{\mathcal B}(L)$. Let ${\mathfrak s}\in{\mathcal B}(G)$ be the corresponding element determined by ${\mathfrak t}$. We recall certain important properties of a $G$-cover $(J,\lambda)$:
\begin{enumerate}
\item[(a)] Let $\overline{P}=L\overline{N}$ be the parabolic subgroup opposite to $P$. Then 
\[
J=J\cap N\cdot J\cap L\cdot J\cap\overline{N},\text{ and }J\cap L=J_L.
\]
The representation $\lambda$ is trivial on $J\cap N$ and $J\cap\overline{N}$, while $\lambda|_{J_L}\cong \lambda_L$.
\item[(b)] The pair $(J,\lambda)$ is an $\mathfrak s$-type in $G$. (cf. \cite[Theorem 8.3]{BK}.)
\item[(c)] There is a canonical injective algebra homomorphism $j_P: {\mathcal H}(L,\lambda_L)\longrightarrow {\mathcal H}(G,\lambda)$ which preserves support of functions and realizes the induction functor ${\iota}_P^G$ under the above said equivalence of categories. (cf. \cite[Corollary 8.4]{BK}.)
\item[(d)] $(J,\lambda)$ is said to be a {\it split cover} if, for every choice of parabolic subgroup $P$ with Levi $L$, the map $j_P$ is an isomorphism of algebras. 
\end{enumerate} 
We note the following lemma whose proof is clear from definitions. 
\begin{lemma}\label{lemma-conj}
Suppose $x\in M$ normalizes the pair $(J_L,\lambda_L)$ and $(J,\lambda)$ is a G-cover of $(J_L,\lambda_L)$. Then the conjugate pair $(J^{x},\lambda^{x})$ is also a $G$-cover of $(J_L,\lambda_L)$. (Here $J^x=xJx^{-1}$ and $\lambda^x(y)=\lambda(x^{-1}yx)$, $y\in J^x$.)
\end{lemma}

In this paper, we are concerned with covers for the {\it general linear group} $G$. Suppose $L=\prod_{i=1}^{k}GL_{n_i}(F)$ is a product of general linear groups and $\sigma=\otimes _{i=1}^{k}\pi_i$, where $\pi_i$ is an irreducible supercuspidal representation of $GL_{n_i}(F)$. By \S\ref{T} each $\pi_i$ contains an extended maximal simple type $(\tilde{J}_i,\tilde{\lambda}_i)$. Put 
\begin{equation}\label{ty}
\tilde{J}_L=\prod_{i=1}^{k} \tilde{J}_i \mbox{ and }\tilde{\lambda}_L=\otimes_{i=1}^{k}\tilde{\lambda}_i,
\end{equation}
then $\sigma\cong\text{c-Ind }\tilde{\lambda}_L$. Let ${\mathfrak t}$ denote the $L$-inertial equivalence class of the pair $(L,\sigma)$. The associated pair $(J_L,\lambda_L)$ is a ${\mathfrak t}$-type. The existence of a $G$-cover of $(J_L,\lambda_L)$ is shown in \cite[\S1.5, Theorem]{Kut2}. The $G$-normalizer $N_G(L)$ of $L$ acts on ${\mathcal B}(L)$ by conjugation. We have the following result of Bushnell and Kutzko regarding the presence of a split cover:

 \begin{proposition}\cite[\S1.5, Theorem]{BK2}\label{split}
If the $N_G(L)/L$-stabilizer of ${\mathfrak t}$ is trivial, then any $G$-cover $(J,\lambda)$ of $(J_L,\lambda_L)$ splits. \end{proposition}

\subsection{Covers in the homogeneous case \cite[\S7]{Kut2}}
\label{c}
Consider $G=GL_{2n}(F)$ and let $B=TU$ denote the $F$-points of the standard Borel subgroup of upper triangular matrices. Let $L=GL_n(F)\times GL_n(F)$, the $F$-points of a maximal Levi subgroup of $G$. To avoid confusion, we write $U_n(F)$ to denote the $F$-points of the upper maximal unipotent radical of $GL_n(F)$. In particular $U\cap L=U_n(F)\times U_n(F)$. Let ${\mathcal A}={\rm End}_F(F^{2n})$ which we identify with $M_{2n}(F)$ after fixing a basis and regard the Levi subgroup $L$ as the stabilizer of a decomposition of $V=F^{2n}$ of the form $V=V_1\oplus V_2$. For convenience, we write $G_i=GL_n(F), i=1,2$. Let $\pi_1,\pi_2$ be irreducible supercuspidal representations of $GL_n(F)$ associated with the same endo-class. Let $\sigma=\pi_1\times \pi_2$ denote the corresponding irreducible supercuspidal representation of $L$. We use the notation of \S\ref{T} by appending subscripts $i$, if
necessary. For instance, we have
\begin{equation}\label{eq1}
\pi_i\cong \mbox{c-Ind}_{\tilde{J}_i}^{GL_n(F)}\tilde{\lambda}_i, i=1,2, 
\end{equation}
where $(\tilde{J}_{i},\tilde{\lambda}_{i})$ is an extended maximal simple type contained in $\pi_i$ and satisfying 
\begin{equation}\label{eq2}
\mbox{Hom}_{U\cap J_i}(\psi,\lambda_i)\neq 0
\end{equation}
as in Proposition~\ref{l1}, where $\lambda_i$ denotes the restriction $\tilde{\lambda}_i|_{J_i}$. 

Due to our assumption on the endo-class, the maximal simple types $(J_i,\lambda_i)$ may be chosen so that they are both associated to a common simple stratum $[\a,k,0,\beta]$ with the same underlying simple character $\theta$ (cf. \cite{DnS}). Whence $J_1=J_2=J(\beta,\a)$, and 
\begin{equation}\label{betaext}
\lambda_i=\kappa\otimes\tau_i, i=1,2, 
\end{equation}
where $\tau_i$ is the lift of a cuspidal representation of $J(\beta,\a)/J^{1}(\beta,\a)\cong GL_{n/d}(k_E)$ with $E=F[\beta]$. Put 
\[
{\sf w}_{0}=\left(\begin{matrix}& I_n\\I_n&\end{matrix}\right),
\]
a representative for a ``certain unique" Weyl group element in $G$. (See \S\ref{LS} below for a precise definition.) Then ${\sf w}_{0}(\sigma)=\pi_2\times\pi_1$.
Let $(J_L, \lambda_L)$ be the corresponding ${\mathfrak t}$-type in $L$ as in (\ref{ty}). We recall the construction of the cover $(J,\lambda)$ in this situation. 
Suppose ${\mathcal L}=\{L_r: r\in \Z\}$ is the lattice chain of period $e$ determined by the order $\a$. It determines lattice chains ${\mathcal L}_i=\{L_i^{j}\}\mbox{ in }V_i, i=1,2$ under the natural identification of $V_i$ with $F^n$. We concatenate these lattice chains together to get a chain of period $2e$ in the sense of \cite[\S 2.8]{Kut2}:
\[
\cdots\supset L_{1}^{0}\oplus L_{2}^{0}\supset L^{0}_{1}\oplus L^{1}_{2}\supset L^{1}_{1}\oplus L^{1}_{2}\supset L^{1}_{1}\oplus L^{2}_{2} \supset \cdots.
\]
This defines a hereditary $\o_F$-order $\a'$ in ${\mathcal A}$ which in (block) matrix form is given by 
\[
\a'=\begin{pmatrix}\a&\a\\\p&\a\end{pmatrix}
\] 
and whose Jacobson radical is 
\[
\p'=\begin{pmatrix}\p&\a\\\p&\p\end{pmatrix}.
\]
We embed $E$ in ${\mathcal A}$ via the map $x \mapsto \begin{pmatrix} x & \\ & x \end{pmatrix}$ and let $\beta'$ denote the image of $\beta$ under this map. Let ${\mathcal B}$ denote the centralizer of $\beta'$ in ${\mathcal A}$ which is isomorphic to $M_{2n/d}(E)$; put $\mathfrak{b}'=\mathfrak{a}'\cap {\mathcal B}$. Thus we obtain a simple stratum $[\a',2k,0,\beta']$ in ${\mathcal A}$ with associated compact open subgroups $H^1(\beta',\a')\subseteq J^1(\beta',\a')\subseteq J(\beta',\a')$ as in \S\ref{T}. By choosing a suitable (ordered) $E$-basis of $V$, we may take the decomposition $V=V_1\oplus V_2$ to be a $E$-decomposition that is {\it subordinate} to the $\o_E$-order $\b'$, in the sense of \cite[Ch. 7]{Kut3}. This ensures that the groups $J(\beta',\a'), J^1(\beta',\a')$ and $H^1(\beta',\a')$ have a Iwahori decomposition with respect to $P$. As in \cite[\S7.2]{Kut2}, set 
\begin{equation}\label{jprime}
  J^{\prime}=(H^1(\beta',\mathfrak{a}') \cap \overline{N}) \times (J(\beta',\mathfrak{a}')\cap P),
\end{equation}
it is a subgroup of $J(\beta',\a')$ containing $H^1(\beta',\a')$. It admits a representation $\lambda'$ of the form $\lambda'=\kappa'\otimes\tau'$, where the restriction of $\kappa'$ to $J'\cap L=J(\beta',\a')\cap L=J(\beta,\a)\times J(\beta,\a)$ is of the form $\kappa\otimes\kappa$ and the $\tau'$ is the inflation of $\tau_1\otimes\tau_2$ to a representation of $J'$. By \cite[Theorem, \S7.2]{Kut2}, the pair $(J',\lambda')$ is a $G$-cover of $(J_L,\lambda_L)$. Fixing a uniformizer $\varpi_E$ of $E$, we have (cf. \cite[Proof of Lemma 4.4]{Modal})
\begin{equation}\label{J-prime}
 J^{\prime} \cap N=\begin{pmatrix} I_n & \varpi_E^{-1}\mathfrak{H}^1(\beta,\mathfrak{a}) \\ & I_n \end{pmatrix} ;\quad 
 J^{\prime} \cap \overline{N}=\begin{pmatrix} I_n & \\ \varpi_E\mathfrak{J}(\beta,\mathfrak{a}) & I_n \end{pmatrix}.
\end{equation}

Since $v_{\mathfrak{a}}(\varpi_E)=1$, it generates the principal ideal $\mathfrak{p}_{\a}$ of $\mathfrak{a}$. For $m\in \Z$ as in \S\ref{sec-ni}, put $x=\begin{psmallmatrix}\varpi_{E}^{m+1}I_n&\\&I_n\end{psmallmatrix}\in L$. It follows from property \ref{B-Hen3} in \S\ref{T} that $x$ normalizes $(J_L,\lambda_L)$. Hence by Lemma~\ref{lemma-conj}, we may conjugate $(J',\lambda')$ by $x$ to obtain the cover $(J'_m,\lambda'_m)$ of $(J_L,\lambda_L)$ satisfying 
\begin{equation}\label{B}
   J'_m\cap N=\begin{pmatrix} I_n & \varpi_E^{m}\mathfrak{H}^1(\beta,\mathfrak{a}) \\ & I_n \end{pmatrix} ;\quad 
   J'_m \cap \overline{N}=\begin{pmatrix} I_n & \\ \varpi^{-m}_E\mathfrak{J}(\beta,\mathfrak{a}) & I_n \end{pmatrix}.
\end{equation}

\begin{remark}\label{level0}
In the level zero case, i.e., $k=0$, we have $m=0$ and (\ref{B}) reduces to 
\[
J'_m\cap N=\begin{pmatrix} I_n & \p \\ & I_n \end{pmatrix} ;\quad 
   J'_m\cap \overline{N}=\begin{pmatrix} I_n & \\ \a & I_n \end{pmatrix},
\]
where $\a=M_{n}(\o_F)$ and $\p=\varpi_F M_{n}(\o_F)$. \end{remark}

\section{The Rankin-Selberg Theory}
\label{sec-RS}

We recall the definition of the $\gamma$-factor attached to pairs, using the formulation of Jacquet, Piatetski-Shapiro, and Shalika in \cite{JPSS3}. Let $\pi_1$ and $\pi_2$ be irreducible admissible (generic) representations of $GL_n(F)$ with associated Whittaker models $\mathcal{W}(\pi_1,\psi)$ and $\mathcal{W}(\pi_2,\overline{\psi})$, respectively. Let $\mathcal{C}_c^{\infty}(F^n)$ be the space of locally constant and compactly supported functions $\Phi : F^n \rightarrow \mathbb{C}$. For each ${\mathcal W}_1 \in \mathcal{W}(\pi_1,\psi)$, ${\mathcal W}_2 \in \mathcal{W}(\pi_2,\overline{\psi})$, and $\Phi \in \mathcal{C}_c^{\infty}(F^n)$, we associate the Rankin-Selberg {\it zeta} integral
\[
  Z(s,{\mathcal W}_1,{\mathcal W}_2,\Phi)=\int\limits_{U_n(F) \backslash GL_n(F)} {\mathcal W}_1(g){\mathcal W}_2(g) \Phi(e_ng)||\mathrm{det}(g)||^s dg,
\]
where $dg$ is a $GL_n(F)$-right invariant measure on $U_n(F) \backslash GL_n(F)$. This integral converges absolutely for $\Re(s)\gg 0$, and it defines a rational function in $\mathbb{C}(q^{-s})$. Let 
\[
  {\sf w}_n=\begin{pmatrix} &&1\vspace{-1ex} \\ &\vspace{-1ex}\iddots& \\ \vspace{-1ex}1 &&  \end{pmatrix}
\]
denote the long Weyl element in $GL_n(F)$. For any smooth representation $(\pi,V)$ of $GL_n(F)$, let $\pi^{\iota}$ denote the representation of $GL_n(F)$ on the same space $V$ given by $\pi^{\iota}(g)=\pi({^tg^{-1}})$. If $\pi$ is irreducible, it is known that $\pi^{\iota}\cong\check{\pi}$, the contragredient representation of $\pi$. If ${\mathcal W} \in \mathcal{W}(\pi,\psi)$, then $\check{\mathcal W}(g):={\mathcal W}({\sf w}_n \; {^tg^{-1}})$ belongs to ${\mathcal W}(\check{\pi},\overline{\psi})$. Let $\widehat{\Phi}$ denote the Fourier transform of $\Phi$ given by
\[
 \widehat{\Phi}(y)=\int\limits_{F^n} \Phi(x) \psi(x\; {^ty} )dx,
\]
where $dx$ is the normalized self-dual measure so that $\widehat{\widehat{\Phi}}(x)=\Phi(-x)$. There is a function $\gamma(s,\pi_1\times\pi_2,\psi)\in {\mathbb C}(q^{-s})$ such that 
\begin{equation}\label{fe}
 Z(1-s,\check{\mathcal W}_1,\check{\mathcal W}_2,\widehat{\Phi})=\omega_{\pi_2}(-1)^{n-1}\gamma(s,\pi_1 \times \pi_2,\psi) Z(s,{\mathcal W}_1,{\mathcal W}_2,\Phi)
\end{equation}
for all $\Phi\in\mathcal{C}_c^{\infty}(F^n)$. Further, the integrals $Z(s,{\mathcal W}_1,{\mathcal W}_2,\Phi)$ span a principal fractional ideal of the ring ${\mathbb C}[q^s,q^{-s}]$ containing $1$. Hence it admits a unique generator of the form $P(q^{-s})^{-1}$ where $P\in {\mathbb C}[X]$ with $P(0)=1$. By definition 
\[
L(s,\pi_1\times\pi_2)=P(q^{-s})^{-1}
\] 
and there is a monomial factor $\varepsilon(s,\pi_1\times\pi_2,\psi)$ of the form $cq^{-f(\pi_1 \times \pi_2,\psi)s}$ so that 
\[
\gamma(s,\pi_1\times\pi_2,\psi)=\frac{\varepsilon(s,\pi_1\times\pi_2,\psi)L(1-s,\check{\pi}_1\times\check{\pi}_2)}{L(s,\pi_1\times\pi_2)}.
\]
Further, the epsilon factor $\varepsilon(s,\pi\times\sigma,\psi)$ satisfies the functional equation
\begin{equation}
\label{RS-FE}
 \varepsilon(1-s,\check{\pi}_1\times\check{\pi}_2,\overline{\psi})\varepsilon(s,\pi_1\times\pi_2,\psi)=1.
\end{equation}

\subsection{The calculation of Paskunas-Stevens}\label{PS}
Here, we briefly review the proof of \cite[Theorem 7.1]{SP} and state that result (see Proposition~\ref{RS-gamma} and Proposition~\ref{RS-gamma-equal} below) in a form suited to this paper. Let $\pi_1$ and $\pi_2$ be unitary supercuspidal representations of $GL_n(F)$ associated to the same endo-class. We then have the extended maximal simple types $(\tilde{J}_i,\tilde{\lambda}_i)$, $i=1,2$, satisfying  (\ref{eq1}) and (\ref{eq2}) with 
\[
\tilde{J}_1=\tilde{J}_2=\tilde{J}(\beta,\a)=E^{\times}J(\beta,\a), E=F[\beta];
\]
and $\lambda_i=\tilde{\lambda}_i|_{J(\beta,\a)}=\kappa\otimes\tau_i, i=1,2$, as in (\ref{betaext}). For $i=1,2$, let ${\mathcal W}_i={\mathcal W}_{\pi_i}\in {\mathcal W}(\pi_i,\psi)$ and let $\check{\mathcal W}_2={\mathcal W}_{\check{\pi}_2}\in {\mathcal W}(\check{\pi}_2,\overline{\psi})$, be the Whittaker functions as in \S\ref{whitt}. Let ${\mathcal J}_i={\mathcal J}_{\tilde{\lambda}_i}$ be the corresponding Bessel function and let $\check{\mathcal J}_i$ denote the Bessel function associated to the dual  of $\tilde{\lambda}_i$. As noted in \cite{SP}, $\check{\mathcal J}_i(g)={\mathcal J}_{\tilde{\lambda}_i}(g^{-1})$, $g\in \tilde{J}(\beta,\a)$. By unitarity, it follows from Proposition~\ref{Whitt} that
\[
\check{\mathcal W}_i=\overline{{\mathcal W}_i}.
\]
By construction $\mathrm{Supp}({\mathcal W}_i) \subset U_n(F)E^{\times}J(\beta,\a)$, $i=1,2$, and 
\[
  {\mathcal W}_1(ug)=\psi(u)\mathcal{J}_1(g), \quad \overline{{\mathcal W}_2(ug)}=\psi^{-1}(u)\overline{\mathcal{J}_2(g)}, u \in U_n(F), g \in \tilde{J}(\beta,\a),
\]
and ${\mathcal W}_1(1)=\overline{{\mathcal W}_2(1)}=1$. Using the pair $({\mathcal W}_1,\check{{\mathcal W}_2})$, one may calculate the zeta integrals on either side of the functional equation (\ref{fe}) for a suitable $\Phi$. In fact, suppose $\Phi=\Phi_{0}$ is the characteristic function on the set $e_n J^1(\beta,\a)$. For any subset $X\supseteq U_n(F)$ of $GL_n(F)$, $\mathrm{vol}_{U_n}(X)$ denotes the volume of $U_n(F) \backslash X$ with respect to a Haar measure $dg$ on $U_n(F) \backslash GL_n(F)$. Also, for any lattice $L$ in $F^n$, we write ${\rm vol}_F(L)$ to denote the volume with respect to the measure $dx$. As shown in \cite[Proposition 7.2]{SP} we have 
\begin{equation}\label{prop-zeta}
  Z(s,{\mathcal W}_1,\check{\mathcal W}_2,\Phi_0)=\mathrm{vol}_{U_n}(U_n(F)H^1(\beta,\a)).
\end{equation}
After certain standard manipulations and using the normalization that the measure on $U_n(F)$ is so that ${\rm vol}(U_n(F)\cap J(\beta,\a))=1$, the corresponding integral on the dual side takes the form 
\begin{equation}\label{shell-sum}
Z(1-s,\check{{\mathcal W}_1},{\mathcal W}_2,\widehat{\Phi}_0)=\sum_{r\in \Z}S_r q_{\a}^{r(s-1)},
\end{equation}
where 
\[
S_r=\int\limits_{J(\beta,\a)}\overline{{\mathcal W}_1(\varpi_E^{r}g)}{{\mathcal W}_2(\varpi_E^{r}g)}\widehat{\Phi}_{0}(e_1\prescript{t}{}{(\varpi_E^{r}g)})dg
\]
and $q_{\a}=q^{n/e}=q_{E}^{n/d}$. 

Considering the support of $\widehat{\Phi}_{0}$ (cf. \cite[Lemma 7.7]{SP}), the sum in (\ref{shell-sum}) is effectively over $r\geq m$ with  
\[
 S_r={\rm vol}_F(e_n\p^{1+m})q^{m}_{\a}\left\{\begin{array}{ll}
  \int\limits_{J(\beta,\a)} \overline{{\mathcal J}_1(\varpi_E^rX)}{{\mathcal J}_2(\varpi_E^rX)}dX&\mbox{if }r>m,\\
  \int\limits_{J(\beta,\a)} \overline{{\mathcal J}_1(\varpi_E^rX)}{{\mathcal J}_2(\varpi_E^rX)}\phi_r(X) dX&\mbox{if }r=m.\end{array}\right.
\]
Here, $\phi_r(X)=\phi(\varpi_E^{r}X)$ is the additive character as in \S\ref{sec-ni}, and $\mathcal{J}_i$ is the Bessel function corresponding to ${\mathcal W}_i$, $i=1,2$. We consider the following two cases. 
 \subsection{The case $\tau_1\ncong \tau_2$} In this case, $\pi_1$ is not equivalent to any unramified twist of $\pi_2$, and by \cite[Lemma 7.10]{SP} we have 
\[
 \int\limits_{J(\beta,\a)} \overline{\mathcal{J}_1( \varpi_E^rX)}  \mathcal{J}_2( \varpi_E^rX) dX=0,
\]
for any $r\in {\mathbb Z}$. Consequently 
\begin{equation}\label{dual-zeta}
 Z(1-s,\check{{\mathcal W}_1},{\mathcal W}_2,\widehat{\Phi}_0)={\rm vol}_F(e_n\p^{m+1})q^{ms}_{\a} \int\limits_{J(\beta,\a)} \overline{{\mathcal W}_1(\varpi_E^mX)}{\mathcal W}_2(\varpi_E^mX)\phi_m(X) dX.
\end{equation}
Combining (\ref{dual-zeta}) and (\ref{prop-zeta}), we obtain the following:
\begin{proposition}
\label{RS-gamma}
Suppose $\tau_1\ncong \tau_2$. Then 
\[
   \gamma(s,\pi_1\times \check{\pi}_2,\psi) =\upsilon \omega_{\pi_2}(-1)^{n-1}q^{ms}_{\a}\int\limits_{J(\beta,\a)} \overline{{\mathcal W}_1(\varpi_E^{m}X)}{\mathcal W}_2(\varpi_E^{m}X) \phi_m(X) dX
\]
with $\upsilon=\dfrac{{\rm vol}_F(e_n\p^{m+1})}{{\rm vol}_{U_n}(U_n(F)H^1(\beta,\a))}$. In particular $f(\pi_1 \times \check{\pi}_2,\psi)=-(mn)/e$.
\end{proposition}
\subsection{The case $\tau=\tau_1\cong \tau_2$}
\label{RS-nonsplit} 
In this situation, there exists an unramified quasi-character $\chi$ of $F^{\times}$ so that $\pi_1\cong \pi_2\otimes(\chi\circ\det)$. If we write $\chi(x)=\|x\|^{s_{0}}, s_{0}\in \mathbb{C}$, then $\gamma(s,\pi_1\times\check{\pi}_2)=\gamma(s+s_{0},\pi_1\times\check{\pi}_1)$. Hence we may assume $\chi$ is trivial, i.e., $\pi=\pi_1\cong \pi_2$. It follows from \cite[(6.2.5)]{Kut3} that 
\begin{equation}
\label{L-equal-RSLS}
L(s,\pi\times \check{\pi})=(1-q_{\a}^{-s})^{-1}.
\end{equation}
Put ${\mathcal W}={\mathcal W}_1={\mathcal W}_2$ and ${\mathcal J}={\mathcal J}_1={\mathcal J}_2$. Let ${\mathcal U}\subset {\mathcal M}\subset \tilde{J}(\beta,\a)$ be as in \S\ref{whitt}. It is shown in \cite[Section 7.3]{SP} that, for $\Re(s)<1$,  
\[
\begin{split}
Z(1-s,\check{{\mathcal W}_1},{\mathcal W}_2,\widehat{\Phi}_0)=\sum_{r\geq m\in \Z}S_r q_{\a}^{r(s-1)}&={\rm vol}_{F}(e_n\p^{m+1})q_{\a}^{ms}{\rm vol}(\mathcal U)\left\{\frac{q_{\a}^s-1}{1-q_{\a}^{s-1}}\right\}\\
&={\rm vol}_{F}(e_n\p^{m+1}){\rm vol}(\mathcal U)q_{\a}^{(m+1)s}\frac{L(1-s,\pi\times\check{\pi})}{L(s,\pi\times\check{\pi})}.\\
\end{split}
\]

By analytic continuation we obtain 
\[
\gamma(s,\pi\times \check{\pi},\psi)= \upsilon {\rm vol}(\mathcal U)\omega_{\pi}(-1)^{n-1}q^{(m+1)s}_{\a}\frac{L(1-s,\pi\times\check{\pi})}{L(s,\pi\times\check{\pi})}.
\]
We apply the functional equation \eqref{RS-FE} to obtain 
\[
1=\upsilon^2{\rm vol}(\mathcal U)^2q_{\a}^{m+1}
\]
which in turn implies $\upsilon{\rm vol}(\mathcal U)=q_{\a}^{-\frac{m+1}{2}}$. 


\begin{proposition}\label{RS-gamma-equal}
For $\pi=\pi_1\cong \pi_2$ as above, we have  
\[
  \gamma(s,\pi\times \check{\pi},\psi)=\omega_{\pi}(-1)^{n-1}q_{\a}^{(m+1)(s-\frac{1}{2})}\frac{L(1-s,\pi\times\check{\pi})}{L(s,\pi\times\check{\pi})}.
 \]
 In particular $f(\pi\times\check{\pi},\psi)=-(m+1)n/e$.
\end{proposition}

\section{The Langlands-Shahidi Local Coefficient} 
\label{LS}
First, we recall the necessary basics of the Langlands-Shahidi method. Let $G,B,P,\ldots$ for the moment be as general as in \S\ref{cov}. Suppose $\Delta$ is a set of simple roots (restricted) in $G$, let $\Delta^L\subset\Delta$ correspond to $L$. Now, assume $P$ is maximal and let $\alpha$ be the unique simple root whose root subgroup belongs to $N$. Let $\widetilde{{\sf w}_{0}}$ be the unique Weyl group element in $G$ such that $\widetilde{{\sf w}_{0}}(\Delta^L)\subset \Delta$ while $\widetilde{{\sf w}_{0}}(\alpha)<0$. We will also assume $P$ is self-associate, i.e., $\widetilde{{\sf w}_{0}}(\Delta^L)=\Delta^L$. It is a standard fact that $X(L)$ is equipped with the structure of a complex torus. This allows us to talk about ``regular" and ``rational" functions of $\chi\in X(L)$ in a certain sense. We write $\iota_{P}^{G}$ to denote the normalized parabolic induction functor  and write ${\mathcal F}_P(\cdot)$ to denote the space of the induced representation $\iota_P^G(\cdot)$. 

For a smooth irreducible representation $\sigma$ of $L$, consider the standard intertwining operator $A(\chi,\sigma,{\sf w}_{0}): {\iota}_P^G(\sigma\otimes\chi)\longrightarrow {\iota}_P^G({\sf w}_0(\sigma\otimes\chi))$ given by 
\begin{equation}\label{eqn-int}
A(\chi,\sigma,{\sf w}_{0})f(g)=\int\limits_N f({\sf w}^{-1}_0ng)dn,\,f\in {\mathcal F}_{P}(\sigma\otimes\chi).
\end{equation}
The integral converges for $\Re(\chi)\gg 0$ and defines a rational function on a non-empty Zariski open subset of the complex torus $X(L)$.  

Now, for $\psi\in \widehat{F}$ as before, it defines a character $\psi^G$ of the maximal unipotent subgroup $U$ of $G$ as explained in \cite[Section 3]{Sha-Ramanujan}. We also pick the representative ${\sf w}_{0}$ so that it is {\it{compatible}} with $\psi^G$, this means that its restriction to $U\cap L$ has the following property:
\[
\psi^G(u)=\psi^G({\sf w}_{0}^{-1}u{\sf w}_{0}), u\in U\cap L. 
\]
Note that $\psi^G$ also determines a character of the maximal unipotent radical $U\cap L$ of $L$ via restriction which we denote as $\psi^L$. Suppose $\sigma$ is generic with respect to this $\psi^L$ and fix a non-zero $\psi^L$-Whittaker functional $\Omega^{L}$ on the space of $\sigma$. For $f\in {\mathcal F}_P(\sigma\otimes\chi)$ such that $\text{Supp}(f)\subset P{\sf w}^{-1}_0N$, define 
\[
\Omega(\chi,\sigma)(f)=\int\limits_N\Omega^L(f({\sf w}^{-1}_0 n) )\overline{\psi^G(n)}dn.
\]
It is well-known that this admits a unique extension to give a non-zero $\psi^G$-Whittaker functional $\Omega(\chi,\sigma)$ on all ${\mathcal F}_{P}(\sigma\otimes\chi)$;  further $\chi\mapsto \Omega(\chi,\sigma)$ is a holomorphic function \cite[Proposition 2.1]{cass2}. As before, for the purpose of calculations, we need \cite[Corollary 2.3]{cass2} which gives a formula for the extension $\Omega(\chi,\sigma)$ in the following sense: Given a compact open subgroup $K$ of $G$, there exists a suitably large compact open subgroup $N_{*}\subset N$ such that 
\[
\Omega(\chi,\sigma)(f)=\int\limits_{N_{*}}\Omega^L(f({\sf w}^{-1}_0 n) ) \overline{\psi^{G}(n)}dn
\]
for all $\chi$ and for all $f\in{\mathcal F}_P(\sigma\otimes\chi)^{K}$. One similarly defines the non-zero functional $\Omega({\sf w}_{0}(\chi),{\sf w}_{0}(\sigma))$ on ${\iota}_P^G({\sf w}_0(\sigma\otimes\chi))$. 

The Langlands-Shahidi local coefficient attached to $\sigma,\psi$ and ${\sf w}_{0}$ is the non-zero constant $C_{\psi}(\chi,\sigma,{\sf w}_{0})$ given by Rodier's multiplicity-one theorem, i.e., 
\begin{equation}
\label{lc}
C_{\psi}(\chi,\sigma,{\sf w}_0)(\Omega({\sf w}_{0}(\chi),{\sf w}_{0}(\sigma))\circ A(\chi,\sigma,{\sf w}_{0}))=\Omega(\chi,\sigma).
\end{equation}

\par
We return to the notation of \S\ref{c} and take $G=GL_{2n}(F), L=GL_n(F)\times GL_n(F)$. Let $\sigma=\pi_1\times\pi_2\in{\mathfrak R}(L)$ be supercuspidal with $\pi_1$ and $\pi_2$ associated to the same endo-class. We then have the extended maximal simple types $(\tilde{J}_i,\tilde{\lambda}_i)$, $i=1,2$, satisfying (\ref{eq1}) and (\ref{eq2}) with 
\[
\tilde{J}_1=\tilde{J}_2=E^{\times}J(\beta,\a);
\]
and $\lambda_i=\kappa\otimes\tau_i, i=1,2$, as in (\ref{betaext}). This determines a ${\mathfrak t}$-type $(J_L,\lambda_L)$ in $L$. For $s\in \mathbb{C}$, let $\chi_s\in X(L)$ be the unramified character given by 
\[
\chi_s(g)=||\mathrm{det}(g_1)||^{s/2}\; ||\mathrm{det}(g_2)||^{-s/2}; \,g=(g_1,g_2)\in L.
\]
With ${\sf w}_0$ fixed as in \S\ref{c}, we note that it is compatible with the character $\psi^G$ of $U$ given by (\ref{psi}), and ${\sf w}_0(\chi_s)=\chi_{-s}$. We write $A(s,\sigma)$ to denote the intertwining operator $A(\chi_{s},\sigma,{\sf w}_0)$ and $C_{\psi}(s,\sigma)$ to denote the corresponding local coefficient $C_{\psi}(\chi_s,\sigma,{\sf w}_0)$. We choose unramified quasi-characters $|| \cdot||^{s_1}$ and $|| \cdot||^{s_2}$ of $F^{\times}$, $s_1, s_2 \in \mathbb{C}$ so that $\pi_1 \otimes || \cdot||^{-s_1}$ and $\pi_2 \otimes || \cdot||^{-s_2}$ are unitary. We put $\pi_1^{\circ}=\pi_1 \otimes || \cdot||^{-s_1}$ and $\pi_2^{\circ}=\pi_2 \otimes || \cdot||^{-s_2}$. One can easily check that
\begin{equation}
\label{Switch-Unitary}
  C_{\psi}(s,\pi_1 \times \pi_2)=C_{\psi}(s+s_1-s_2,\pi^{\circ}_1 \times \pi^{\circ}_2).
\end{equation}
Hence for calculation purposes, we may assume that both $\pi_1$ and $\pi_2$ are unitary.

\subsection{A note on measures}\label{sec-measures} Clearly the definition of $A(\chi,\sigma,{\sf w}_{0})$ and $\Omega(\chi,\sigma)$ involves a choice of Haar measure $dn$ on $N$. Following \cite[\S 5.2]{BHK}, for a random measure $dn$ on $N$, we always choose the measure $d\overline{n}$ on $\overline{N}$ that is dual to $dn$, relative to $\psi$. Then the measure $dn\otimes d\overline{n}$ on $N\times \overline{N}$ is independent of the initial choice of $dn$. (It only depends on $L$ and the additive character $\psi$.) Hence, for any compact open subgroup $K\leq G$, the product of volumes ${\rm vol}(K\cap N){\rm vol}(K\cap\overline{N})$ is independent of the choice of $dn$. We will exploit this fact in our calculations with $K=J'$ as in (\ref{jprime}).


\subsection{\bf The case $\tau_1\ncong \tau_2$} \label{sec:unequal}
Let $(J,\lambda)=(J'_m,\lambda'_m)$ be the cover as in (\ref{B}).
In this situation, we have  
\[
\pi_2\not\cong \pi_1\otimes(\chi\circ\det),
\] 
for any unramified character $\chi$ of $F^{\times}$, which in turn implies that the cover $(J,\lambda)$ splits (cf. Proposition~\ref{split}). We introduce certain functions in the induced representation space for later use. Let $K$ be any compact open subgroup of $\overline{N}$ and let $V_{\sigma}$ denote the space of $\sigma$. For $u\in V_{\sigma}$, consider the function $f_u=f_{u,K}\in \mathcal{F}_{P}(\sigma)$ defined as follows (cf. \cite[\S 1.2]{BHK}): $f_u$ is supported in $PK$ and 
\begin{equation}
\label{f}
f_u(xk)=\delta_{P}(x)^{1/2}\sigma(x)u, x\in P, k\in K.
\end{equation}
For any compact open subgroup $K$ of $N$, similarly define the function $f'_u=f'_{u,K}\in {\mathcal F}_{P}({\sf w}_0(\sigma))$ supported in $P{\sf w}_0K$ and given as
\begin{equation}
\label{f'}
f'_u(x{\sf w}_0k)=\delta_{P}(x)^{1/2}{\sf w}_0(\sigma)(x)u, x\in P, k\in K.
\end{equation}

\par
Let us re-write (\ref{f}) and (\ref{f'}) in the current context. Let $W_i$ denote the representation space of $\tilde{\lambda}_i$, $i=1,2$; thus the representation space $W$ of $\tilde{\lambda}_{L}$ is given by $W=W_1\otimes W_2$. Since each $\lambda_i$ contains $\psi$, it follows that $\lambda_L$ contains the character $\psi^L=\psi\times\psi$ of $U\cap L=U_n(F)\times U_n(F)$. Further, there is a canonical $\tilde{J}_i$-homomorphism from $W_i$ into the space of $\pi_i$ given by $w_i\mapsto \varphi_{w_i}$, where $\varphi_{w_i}$ is supported in $\tilde{J}_i$. Given $w_1\in W_1$ and $w_2\in W_2$, let $w=w_1\otimes w_2$ and let $\varphi_w=\varphi_{w_1}\otimes \varphi_{w_2}$. One checks that $\varphi_w$ belong to the $\lambda_L$-isotypic subspace of $\pi$. Let $f_w\in {\mathcal F}_P(\sigma\otimes\chi_s)$ and $f'_w\in {\mathcal F}_P({\sf w}_0(\sigma\otimes\chi_s))$ denote the function $f_{\varphi_w}$ and $f'_{\varphi_w}$, respectively. 

\begin{remark}
It will be helpful to relate our notation to that of \cite{BHK}. For instance, the function $f_w=\text{vol}(J\cap\overline{N})\,f_{\varphi_w}^u$, where $f_{\varphi_w}^u$ is the function defined in  \cite[\S1.3]{BHK}. On the other hand, since ${\sf w}_0\overline{P}{\sf w}_0^{-1}=P$, for any smooth representation $\tau$ of $L$, we see that the map $L({\sf w}_0): i_{\overline{P}}^G(\tau)\longrightarrow i_{P}^{G}({\sf w}_0(\tau))$, given by $f\mapsto f({\sf w}_0^{-1}\cdot)$, is a $G$-isomorphism. Thus if $f_{\varphi_w}^{l}\in i_{\overline{P}}^{G}(\sigma\otimes\chi_s)$ is as in \cite[\S1.2]{BHK}, then $f'_w=\text{vol}(J\cap N)\,L({\sf w}_0)f_{\varphi_w}^{l}$. 
\end{remark}
 \begin{proposition}\label{lemma-1}
If the $G$-cover $(J,\lambda)$ of $(J_L,\lambda_L)$ splits, then  
\[
A(s,\sigma)f_w=\mathrm{vol}(J\cap\overline{N}) f'_w \quad \text{and} \quad A(-s,{\sf w}_0(\sigma)) \circ A(s,\sigma)f_w=\mathrm{vol}(J\cap N)\mathrm{vol}(J\cap\overline{N})f_w.
\] 
\end{proposition}
\begin{proof}
This is a reformulation of \cite[Proposition 2.4]{BHK}. Namely, assuming that the measures are normalized so that $\text{vol}(J\cap N)=\text{vol}(J\cap\overline{N})=1$, it is shown in loc.cit. that $A(s,\sigma)f_{\varphi_w}^{u}=f_{\varphi_w}^{l}$. However, with the above remark in mind, this in turn implies our assertion.
\end{proof}

 We fix $\psi$-Whittaker functionals $\Omega_i$ for $\pi_i$, $i=1,2$,  as in \S\ref{whitt}; then $\Omega^{L}=\Omega_1\otimes\Omega_2$ is a $\psi^L$-Whittaker functional for $\sigma=\pi_1\times \pi_2$. Let $\Omega(s,\sigma)$ denote $\Omega({\chi_s},\sigma)$ defined with respect to $\psi^G$. Before we proceed further, writing a typical element of $N$ as $n(X)=\left(\begin{matrix}I_n&X\\&I_n\end{matrix}\right)$, we note that the character 
 \[
 X\mapsto \psi^G(n(X))=\psi(X_{n1})
 \]
 is nothing but the additive character $\phi$ of $A$ introduced in \S\ref{sec-ni}.
 
  \begin{lemma}\label{lemma-2}
Keeping the above notation, we have 
\[
\Omega(-s,{\sf w}_0(\sigma))(f'_w)=\Omega^L(\varphi_w)\mathrm{vol}(J\cap N).
\]
Further, we may choose $w$ so that $\Omega^L(\varphi_w)=1$.
\end{lemma}
\begin{proof}
This is easy since the function $f'_w$ is supported on $P{\sf w}_0N$. Namely,
\[
\Omega(-s,{\sf w}_0(\sigma))(f'_w) =  \int\limits_{J\cap N}\Omega^L(f'_w({\sf w}_0n)\overline{\psi^G(n)}dn=\Omega^L(\varphi_w)\int\limits_{J\cap N}\overline{\psi^G(n)}dn.
\]
Since 
\[
n=n(X)\in N\cap J\implies X\in \p^{m+1},
\]
it follows that $\psi^G(n(X))=\phi(X)=1$ for $u=n(X)\in J\cap N$. Consequently
\[
\Omega(-s,{\sf w}_0(\sigma))(f'_w)=\Omega^L(\varphi_w)\text{vol}(J\cap N).
\] 
For $i=1,2$, let ${\mathcal W}_i={\mathcal W}_{\pi_i}$ be the Whittaker function defined in \S\ref{whitt}. By construction ${\mathcal W}_i={\mathcal W}_{\varphi_{w_i}}$ for a unique $w_i\in W_i$. Then $\Omega^L(\varphi_w)={\mathcal W}_1(1){\mathcal W}_2(1)=1$ for the corresponding $w=w_1\otimes w_2$. 
\end{proof}

This brings us to the central issue of computing $\Omega(s,\sigma)(f_w)$. It is not straightforward since $f_w$, unlike $f'_{w}$, is supported near the identity element. Let us fix $w=w_1\otimes w_2$ as in the above lemma so that ${\mathcal W}_i={\mathcal W}_{\varphi_{w_{i}}}={\mathcal W}_{\pi_i}, i=1,2$. For $X\in GL_n(F)$, let $n(X)\in N$ be as in the proof above and let $\bar{n}(X)\in\overline{N}$ denote the element $\bar{n}(X)=\left(\begin{matrix}I_n&\\X&I_n\end{matrix}\right)$. 
We have 
\[
{\sf w}_{0}n(X)=\left(\begin{matrix}-X^{-1}&\\&X\end{matrix}\right)\left(\begin{matrix}I_n&-X\\&I_n\end{matrix}\right)\left(\begin{matrix}I_n&\\X^{-1}&I_n\end{matrix}\right),
\]
hence 
\[
f_w({\sf w}_0n(X))=\left\{ \begin{array}{ll}\text{det}(X)^{-s-n}(\pi_1(-X^{-1})\otimes\pi_2(X)\cdot \varphi_w)&\mbox{ if }\bar{n}(X^{-1})\in J\cap\overline{N}\\0&\mbox{otherwise.}\end{array}\right.
\]
Since
\[
   J \cap \overline{N}=\begin{pmatrix} I_n & \\ \varpi^{-m}_E\mathfrak{J}(\beta,\mathfrak{a}) & I_n \end{pmatrix},
\]
it follows that $\Omega(s,\sigma)(f_w)$ equals 
\[
\int\limits_{\left\{\substack{X\in \p^{-t}:\text{det}(X)\neq 0\\X^{-1}\in \varpi_E^{-m}{\mathfrak J}(\beta,\a)}\right\}}\|\text{det}(X)\|^{-s-n}\Omega^{L}(\pi_1(-X^{-1})\otimes\pi_2(X)\cdot \varphi_w)\overline{\phi(X)}dX
\]
for some $t\gg 0$. 

Let us write $\Omega^{L}(\pi_1(-X^{-1})\otimes\pi_2(X)\cdot \varphi_w)$ as ${\mathcal W}_1(-X^{-1}){\mathcal W}_2(X)$,
where, for $i=1,2$, ${\mathcal W}_i={\mathcal W}_{\pi_i}={\mathcal W}_{{\varphi_{w_i}}}$ as in the above lemma. Then 
\[
\begin{aligned}
\Omega(s,\sigma)(f_w)=\int\limits_{\left\{\substack{X\in \p^{-t}:\text{det}(X)\neq 0\\X^{-1}\in \varpi_E^{-m}{\mathfrak J}(\beta,\a)}\right\}}\|\text{det}(X)\|^{-s}{\mathcal W}_1(-X^{-1}){\mathcal W}_2(X)\overline{\phi(X)}d^{\times}X,
\end{aligned}
\]
where $d^{\times}X=\frac{dX}{\|\text{det}(X)\|^n}$ which is invariant for the adjoint action of $L$ on $N$. Making the change of variable $X \mapsto -X$, we obtain 
\[
\Omega(s,\sigma)(f_w)=\omega_{\pi_2}(-1)\int\limits_{ \left\{\substack{X\in \p^{-t}:\text{det}(X)\neq 0\\X^{-1}\in \varpi_E^{-m}{\mathfrak J}(\beta,\a)}\right\}}
   {\mathcal W}_1(X^{-1}){\mathcal W}_2(X) \phi(X) ||\mathrm{det}(X)||^{-s}d^{\times}X
\]
for sufficiently large $t$. Since the Whittaker functions ${\mathcal W}_i,\; i=1,2$, are both supported in $U_n(F)E^{\times}J(\beta,\a)$, it is enough to consider those $X$ which belong to the set
\[
  \bigcup_{-t \leq r \leq m} U_n(F) \varpi_E^r J(\beta,\a).
\]
For any integer $r$, define the {\it shell} $\mathcal{D}_r$ by
\[
 {\mathcal D}_r= \{ X \in U_n(F)\varpi_E^r J(\beta,\a) :  X \in \mathfrak{p}^{-t}, X^{-1} \in \varpi^{-m}_E\mathfrak{J}(\beta,\mathfrak{a}) \},
\]
and re-write the above integral as 
\[
\Omega(s,\sigma)(f_w)= \omega_{\pi_2}(-1)\sum_{-t \leq r \leq m}  q_{\a}^{rs}I_r, 
\]
where 
\[
 I_r=\int\limits_{\mathcal{D}_r}  {\mathcal W}_1(X^{-1}){\mathcal W}_2(X) \phi(X) d^{\times} X.
 \]
Using Proposition~\ref{lemma-1} and Lemma~\ref{lemma-2} we obtain 
\begin{equation}
\label{sum-shell}
\mathrm{vol}(J\cap N)  {\rm vol}(J\cap\overline{N}) C_{\psi}(s,\pi_1\times\pi_2)=\omega_{\pi_2}(-1)\sum_{-t \leq r \leq m}  q_{\a}^{rs}I_r.
\end{equation}
On the other hand \cite[Proposition 2.1(b)]{Sha10} combined with the above expression implies that $C_{\psi}(s,\pi_1\times\pi_2)\in \mathbb{C}[q^s,q^{-s}]^{\times}$, i.e., a monomial in $q^{-s}$. Therefore all but one of the integrals $I_r$ must vanish. We claim $I_{m}\neq 0$, completing the computation of the local coefficient in this case. To that end, we require the following lemma.

\begin{lemma}
\label{shell-indetification}
One has $\mathcal{D}_{m}=\varpi_E^{m}J(\beta,\a)$.
\end{lemma}

\begin{proof}
We will use the fact that $\varpi_E$ normalizes the $\o_F$- order ${\mathfrak J}(\beta,\a)$ and hence its unit group $J(\beta,\a)={\mathfrak J}(\beta,\a)^{\times}$. It is clear that $\varpi_E^{m}J(\beta,\a) \subset \mathcal{D}_{m}$. For the reverse inclusion, suppose $X \in \mathcal{D}_{m}$, write
\[
   X=u\varpi_E^{m}j, \quad u \in U_n(F), j \in J(\beta,\a).
\]
Since $X^{-1} \in \varpi^{-m}_E\mathfrak{J}(\beta,\a)$, it follows that $u^{-1} \in J(\beta,\a)$ which implies that $X \in \varpi_E^{m} J(\beta,\a)$. 
\end{proof}

It follows from Lemma \ref{shell-indetification} that
\[
  I_{m}=\int\limits_{J(\beta,\a)}  {\mathcal W}_1((\varpi_E^{m}X)^{-1}){\mathcal W}_2(\varpi_E^mX) \phi_m(X) dX,
\]
which is precisely the integral in the expression for $\gamma(s,\pi_1\times\check{\pi}_2,\psi)$ in Proposition~\ref{RS-gamma}. Hence $I_m\neq 0$ and we obtain the following theorem.

\begin{theorem}
\label{LS-equal}
Suppose $\pi_1 \cong \mathrm{c}\text{-}\mathrm{Ind}^{GL_n(F)}_{E^{\times}J(\beta,\a)}(\tilde{\lambda}_1)$ and $\pi_2 \cong  \mathrm{c}\text{-}\mathrm{Ind}^{GL_n(F)}_{E^{\times}J(\beta,\a)}(\tilde{\lambda}_2)$ are unitary supercuspidal representations of $GL_n(F)$ associated to the same endo-class and $(J,\lambda)$ is the corresponding cover as above. Assume $\pi_2 \not\cong \pi_1 \otimes(\chi\circ\det)$ for any unramified character $\chi$ of $F^{\times}$. Then $I_r=0$ for $r < m$ and
\[
 C_{\psi}(s,\pi_1 \times \pi_2)=\upsilon^{-1}\mathrm{vol}(J\cap N)^{-1}\mathrm{vol}(J\cap\overline{N})^{-1}\omega_{\pi_2}(-1)^n \gamma(s,\pi_1 \times \check{\pi}_2,\psi).
\]
\end{theorem}


\subsection{\bf The case $\tau_1 \cong \tau_2$}
\label{sec:equal}
In this case, as noted in \S\ref{RS-nonsplit}, we may take $\pi_1\cong \pi_2$ which henceforth will be denoted as $\pi$. Throughout \S\ref{sec:equal}, we do not assume $\pi$ is unitary. Let $\tau$ denote the representation $\tau_1\cong \tau_2$. We follow the paradigm of \cite[\S4.2]{BHK} to reduce the calculation of the local coefficient to that associated with a unramified principal series representation of $GL_2(\mathfrak k)$, where ${\mathfrak k}$ is an unramified field extension of $E$ (see below). In this latter situation, the local coefficient was first computed by Casselman \cite{cass2} which we review below. (See \cite{KK} for a treatment of  ramified principal series representations via the theory of types and covers.) Let us first collect certain properties of the cover $(J',\lambda')$ (cf. \S\ref{c}) in this situation. 

Fix a maximal simple type $(J_{1},\lambda_1)$ associated to a simple stratum $[\a,k,0,\beta]$ contained in $\pi$ as in Proposition~\ref{l1}. Let $(J_L,\lambda_L)$ be the corresponding ${\mathfrak t}$-type in $L$, namely, $J_L=J_1\times J_1$ and $\lambda_{L}=\lambda_1\times\lambda_1$. Recall associated to this data is another simple stratum $[\a',2k,0,\beta']$ in ${\mathcal A}=M_{2n}(F)$ with associated subgroups $H^1(\beta',\a')\subseteq J^{1}(\beta',\a')\subseteq J(\beta',\a')$. By definition $J'$ is a subgroup of $J(\beta',\a')$ containing $H^1(\beta',\a')$. According to \cite[Theorem (7.2.17)]{Kut3} there is a representation $\eta$ of $J(\beta',\a')$ so that $\lambda'$ is the natural representation of $J'$ on the space of $J^1(\beta',\a')\cap N$-fixed vectors in $\eta$ and $\eta={\rm Ind}_{J'}^{J(\beta',\a')}{\lambda'}$. In the terminology of loc.cit., the parabolic subgroup $P$ is {\it subordinate} to $(J(\beta',\a'),\eta)$ and $(J_1,\lambda_1)$ is the {\it associated} maximal simple type of $(J(\beta',\a'),\eta)$. 
This entails a support preserving isomorphism of Hecke algebras ${\mathcal H}(G,\lambda')\cong {\mathcal H}(G,\eta)$, meaning, if $\phi'\in {\mathcal H}(G, \lambda')$ has support  $J'gJ'$ for some $g\in G$, then its image $\phi\in {\mathcal H}(G,\eta)$ has support $J(\beta',\a')gJ(\beta',\a')$, and the space of functions supported on a double coset is one dimensional (cf. \cite[7.2.19]{Kut3}). By \cite[(5.6.6)]{Kut3}, ${\mathcal H}(G,\eta)$ is a affine Hecke algebra.

Let ${\mathfrak k}$ be the unramified extension of degree $n/d$ over $E$ with ${\mathfrak k}^{\times}\subset {\mathfrak K}(\b)$. Put ${\mathcal C}={\rm End}_{\mathfrak k}(V)$, ${\mathfrak c}'={\a'}\cap {\mathcal C}=\b'\cap{\mathcal C}$, and choose the decomposition $V=V_1\oplus V_2$ so that it is a ${\mathfrak k}$-decomposition subordinate to the $\o_{\mathfrak k}$-order $\c'$. Select a $\o_{\mathfrak k}$-basis of the lattice chain ${\mathcal L}$ so that $G':={\mathcal C}^{\times}$ is identified with $GL_2(\mathfrak k)$ and $\c'$ is identified with 
\[
\c'=\begin{pmatrix}\o_{\mathfrak k}&\o_{\mathfrak k}\\\p_{\mathfrak k}&\o_{\mathfrak k}\end{pmatrix}.
\] 
Utilizing this basis construct a $\o_E$-basis of the lattice chain ${\mathcal L}$ as in \cite[(5.5.2)]{Kut3} so that $V=V_1\oplus V_2$ is also a $E$-decomposition that is subordinate to $\a'$. With this configuration, the following properties hold \cite[(7.6.17)]{Kut3}: 
\begin{enumerate}
\item $J(\beta',\a')\cap G'={\mathcal I}'$, where ${\mathcal I}'(=U(\c'))=\begin{psmallmatrix}\o_{\mathfrak k}^{\times}&\o_{\mathfrak k}\\{\p}_{\mathfrak k}&\o_{\mathfrak k}^{\times}\end{psmallmatrix}$ is the standard Iwahori subgroup of $GL_2(\mathfrak k)$;
\item $P\cap G'=B'$ is the standard Borel subgroup of $G'$ and is subordinate to the simple type $({\mathcal I}',1_{{\mathcal I}'})$ in $G'$ ($1_{{\mathcal I}'}$ is the trivial character of ${\mathcal I}'$);
\item $L\cap G'=A'\cong {\mathfrak k}^{\times}\times {\mathfrak k}^{\times}$ is the diagonal torus in $G'$ and $B'=A'U'$, $U^{\prime} =N\cap G'$ is the upper triangular matrix consisting of $1$'s on the diagonal.
\end{enumerate}

Let $K^{\prime}= GL_2(\mathfrak{o}_{\mathfrak k})$ denote the maximal compact subgroup and let $A_0^{\prime}=A^{\prime} \cap K^{\prime}$. Let ${\sf W}=N_{G^{\prime}}(A^{\prime})/A^{\prime}$ be the Weyl group and let $\widetilde{\sf W}=N_{G^{\prime}}(A^{\prime})/{A_0^{\prime}}$ denote the affine Weyl group. Let $X_{\ast}=X_{\ast}(A')$ denote the group of cocharacters of $A'$. It is a free abelian group of rank $2$ and is identified with ${\mathbb Z}^2$ as follows: to the pair $(m,n)$ corresponds the cocharacter which sends $z$ in the multiplicative group ${\mathfrak k}^{\times}$ to the diagonal matrix $\begin{psmallmatrix}z^m&\\&z^n\end{psmallmatrix}$. We write $\Lambda^{\prime}_{+}$ to denote the set of dominant weights given by ${A_{+}^{\prime}}\slash {A_0^{\prime}}$ where 
\[
{A_{+}^{\prime}}=\{ a \in A^{\prime} \; : \; a(\mathcal{I}^{\prime} \cap U^{\prime})a^{-1} \subseteq (\mathcal{I}^{\prime} \cap U^{\prime}) \}.
\]
It corresponds to the set of all pairs $(m,n)$ satisfying $m\geq n$. There is a canonical isomorphism $A'/A'_{0}\cong X_{\ast}$ under which a cocharacter $\mu$ corresponds to the class of the element $\mu(\varpi_{\mathfrak k})$. Consequently $\widetilde{\sf W}=X_{\ast} \rtimes \sf W$. 

\par
We note that elements in $\widetilde{\sf W}$ may be viewed as elements in $G$ since $G'$ is a subgroup of $G$.  
Then ${\sf W}=\{1, {\sf w}_0\}$. It is a well known fact that $\widetilde{\sf W}$ can also be viewed as an extension of a Coxeter group: Namely, put 
\[
   t=\begin{pmatrix}  & 1 \\ \varpi_{\k} & \end{pmatrix}, s_{0}=\begin{pmatrix}  & \varpi_{\k}^{-1} \\ \varpi_{\k} & \end{pmatrix},\mbox{and }s_1={\sf w}_0;
  \] 
then $ts_1t^{-1}=s_{0}$ and $s_{0}^2=s_1^2=(s_0s_1)^3=1$. Let $R\leq \widetilde{\sf W}$ be the subgroup generated by the elements $\{s_{0},s_{1}\}$; it is a Coxeter group. One has $\widetilde{\sf W}=\langle t \rangle \rtimes R$. Hence every ${\sf w} \in \widetilde{\sf W}$ has a unique expression 
\begin{equation}
\label{H-reduced}
{\sf w}=t^k s_{j_1}s_{j_2}\ldots s_{j_{\ell}}, \mbox{with }s_{j_i}\in \{s_{0},s_{1}\},
\end{equation}
where $\ell=\ell({\sf w})$ is the smallest number of $s_j$ needed and is called the length of ${\sf w}$. It satisfies the formula 
\[
 q_{\a}^{\ell({\sf w})}=[{\mathcal I}'{\sf w}{\mathcal I}': {\mathcal I}']=[{\mathcal I}': \mathcal{I}^{\prime}\cap {\sf w}\mathcal{I}^{\prime}{\sf w}^{-1}].
\]
(Observe that $q_{\a}=[\mathfrak{o}_{\k} : \mathfrak{p}_{\k} ]$.)

Suppose the Haar measure on $G'$ is so that ${\rm vol}({\mathcal I}')=1$. For ${\sf w} \in \widetilde{\sf W}$, let ${\vartheta}'_{\sf w}\in {\mathcal H}(G',1_{\mathcal I}')$ denote the characteristic function of the double coset ${\mathcal I}'{\sf w}{\mathcal I}'$. For a cocharacter $\mu$, this is interpreted as the characteristic function of the double coset ${\mathcal I}'\mu(\varpi_{\mathfrak k}){\mathcal I}'$. The collection $\{\vartheta'_{\sf w}\}_{{\sf w} \in \widetilde{\sf W}}$ forms a ${\mathbb C}$-basis for the algebra ${\mathcal H}(G',1_{\mathcal I}')$ and the following relations are known:
\begin{eqnarray}
\vartheta'_{{\sf w}_{1}}\star \vartheta'_{{\sf w}_{2}}&=&\vartheta'_{{\sf w}_{1}{\sf w}_{2}} \mbox{ if }\ell({\sf w}_{1}{\sf w}_{2})=\ell({\sf w}_{1})+\ell({\sf w}_{2}),\label{length}\\
\vartheta'_{{\sf w}_{0}}\star \vartheta'_{{\sf w}_{0}}&=&(q_{\a}-1)\vartheta'_{{\sf w}_{0}}+q_{\a}\vartheta'_{1}\label{quad-prime}.
\end{eqnarray}

If $1_{\mathfrak k}$ is the trivial character of ${\mathfrak k}^{\times}$, then $1_{\mathfrak k}\times 1_{\mathfrak k}$ is a supercuspidal representation of $A'$ whose associated type is given by the compact open subgroup $A'_{0}\cong {\mathfrak o}_{\mathfrak k}^{\times}\times {\mathfrak o}_{\mathfrak k}^{\times}$ and its trivial character. The pair $({\mathcal I}', 1_{{\mathcal I}'})$, which is a simple type in $G'$, is a $G'$-cover of $(A'_{0},1_{A'_{0}})$ whose associated maximal simple type is $({\o}^{\times}_{\k}, 1_{{\o}^{\times}_{\k}})$. We identify ${\mathcal H}(A', 1_{A'_{0}})={\mathcal H}(\k^{\times}, 1_{\o^{\times}_{\k}})\otimes{\mathcal H}({\k}^{\times},1_{\o^{\times}_{\k}})$ and ${\mathcal H}(L,\lambda_L)={\mathcal H}(GL_n(F),\lambda_1)\otimes {\mathcal H}(GL_n(F),\lambda_1)$. By \cite[(7.6.20)]{Kut3}, there are canonical algebra isomorphisms $\Psi: {\mathcal H}(G',1_{{\mathcal I}'})\longrightarrow {\mathcal H}(G,\lambda')$ and $\Psi': {\mathcal H}(A',1_{A'_{0}})\longrightarrow {\mathcal H}(L,\lambda_{L})$ so that the following diagram is commutative:
\begin{equation}\label{fig-ind}
\xymatrix{
{\mathcal H}(A',1_{A'_{0}}) \ar[d]_{j_{B'}} \ar[r]^{\Psi'}
& {\mathcal H}(L,\lambda_L)  \ar[d]^{j_P} \\
{\mathcal H}(G',1_{{\mathcal I}'}) \ar[r]_{\Psi} & {\mathcal H}(G,\lambda'). }
\end{equation}
where the vertical maps $j_{B'}$ and $j_P$ realize the (normalized) induction functor  ${\iota}_{B'}^{G'}$ and ${\iota}_{P}^{G}$, respectively, as explained in loc.cit. To keep notations short, we write ${\mathcal H}$ to denote Hecke algebra ${\mathcal H}(G,\lambda')$ and ${\mathcal H}'$ to denote the Hecke algebra ${\mathcal H}(G',1_{{\mathcal I}'})$. 

Observe that the Hecke algebra ${\mathcal H}$ is supported on $J'\widetilde{\sf W}J'$. The representation $\lambda_L$ extends to a representation of $A^{\prime}$ via $\widetilde{\lambda}_L$ (since ${\mathfrak k}^{\times}\subset \widetilde{J}_1$) and to a representation of $\sf W$ via permutation of vectors in tensor products. We denote this extended representation of $\widetilde{\sf W} \ltimes J_L$ as $\dot{\lambda}_{L}$. For ${\sf w} \in \widetilde{\sf W}$, let $\vartheta_{\sf w} \in \mathcal{H}$ denote the (normalized) element so that 
\[
  \mathrm{Supp}(\vartheta_{\sf w})=J'{\sf w}J'\quad \text{with} \quad \vartheta_{\sf w}({\sf w})={\rm vol}(J')^{-1}v_{\sf w}\dot{\lambda}_L({\sf w}),
\]
where $v_{\sf w}=\left(\frac{{\rm vol}({\mathcal I}'{\sf w}{\mathcal I}')}{{\rm vol}_{J'}(J'{\sf w}J')}\right)^{1/2}$ with ${\rm vol}_{J'}(J'{\sf w}J')=[J'{\sf w}J':J']$, a factor  that is independent of any Haar measure on $G$. Let $a\in \widetilde{\sf W}$ denote the element $a=\begin{pmatrix}\varpi_{\k}&\\&1\end{pmatrix}$. In general, there is an ambiguity up to a scalar in describing an isomorphism between ${\mathcal H}' \cong {\mathcal H}$, the isomorphism $\Psi$ above is fixed (see \cite[(7.6.24)]{Kut3}) by stipulating 
\begin{equation}\label{eqn-a}
\Psi(\vartheta'_{a})=\vartheta_{a}.
\end{equation}
(The normalizing factors ``$\delta_P$" and ``$\delta_Q$" in loc.cit. are absorbed as ``$v_{\sf w}$" in the definition of $\vartheta_{\sf w}$.)

\par
In what follows, we write $\Sigma_{\sf w}$ to denote the set of right cosets 
\[
\Sigma_{\sf w}=J'/(J'\cap J'^{{\sf w}}), {\sf w} \in \widetilde{\sf W}.
\]
We need to determine the effect of $\Psi$ on all basis elements $\vartheta_{\sf w} , {\sf w} \in \widetilde{\sf W}$. To that end, we start with the following:

\begin{lemma}\label{lemma-s1}
Keep the above notation. Then, for $s:=s_1={\sf w}_0$, we have  
\[
\Psi(\vartheta'_{s})=\omega_{\pi}(-1)\vartheta_{s}.
\]
 \end{lemma}

 \begin{proof}
 Since the space of functions supported on a double coset is one dimensional, we have $\Psi(\vartheta'_{s})=c_s \vartheta_{s}$ for some scalar $c_s \in {\mathbb C}^{\times}$. The quadratic relation $\vartheta'_{s}\star \vartheta'_{s}=q_{\a}\vartheta'_1+(q_{\a}-1)\vartheta'_{s}$ implies that 
  \begin{equation}\label{eqn-w0}
c_s^2(\vartheta_{s}\star \vartheta_{s})= \Psi(\vartheta'_{s}\star \vartheta'_{s})=q_{\a}\vartheta_1+c_s(q_{\a}-1)\vartheta_{s}.
 \end{equation}
 \par
On the other hand 
\[
(\vartheta_{s}\star \vartheta_{s})(s)=\int\limits_{J'sJ'}\vartheta_{s}(x) \vartheta_{s}(x^{-1}s)dx={\rm vol}(J') \int\limits_{J'sJ'/J'}\vartheta_{s}(x) \vartheta_{s}(x^{-1}s)dx 
= v_{s}\sum_{j\in \Sigma_s}\dot{\lambda}_{L}(s) \vartheta_{s}(sj^{-1}s).
\]
Using the Iwahori factorization of $J'$, one sees that the map $n(x)=\left(\begin{smallmatrix}I_n&x\\&I_n\end{smallmatrix}\right)\mapsto x$ gives a bijection between 
\[
\Sigma_s\cong \varpi_{E}^{-1}{\mathfrak H}^1(\beta,\a)/\varpi_E{\mathfrak J}(\beta,\a),
\]
and that $J'sJ'=J's(J'\cap N)$. Let us write $j_x$ to denote the matrix $\left(\begin{smallmatrix}I_n&x\\&I_n\end{smallmatrix}\right)$. Writing $j=n(y)$ in the above summand, we see that 
\begin{equation}\label{imply1}
sn(-y)s\in J'sJ' \Leftrightarrow sn(-y)sn(x)s\in J'
\end{equation}
for some $x\in \varpi^{-1}_E{\mathfrak H}^1(\beta,\a)$. Using the Bruhat decomposition (for $y\neq 0$)
 \[
   s\begin{pmatrix} I_n & -y \\ & I_n \end{pmatrix} s=\begin{pmatrix} y^{-1} & I_n \\ & -y \end{pmatrix} s \begin{pmatrix}I_n & -y^{-1} \\ & I_n  \end{pmatrix},
  \]
  we see that 
  \[
  (\ref{imply1})\Leftrightarrow \begin{pmatrix}y^{-1}&I_n\\&-y\end{pmatrix}\begin{pmatrix}I_n&\\x-y^{-1}&I_n\end{pmatrix}\in J'
  \]
  which in turn implies that $x,y\in {\mathfrak J}(\beta,\a)^{\times}=J(\beta,\a)$. Hence, for $y$ such that $sj_ys$ belongs to the support of $\vartheta_{s}$, 
  \[
  \vartheta_{s}(sj_y s)={\rm vol}(J')^{-1}v_{s}\omega_{\pi}(-1)(\lambda_1(y^{-1})\otimes \lambda_1(y))\dot{\lambda}_{L}
(s) 
 \]
Now, by choosing an eigenbasis for the operator $\lambda_1(y)$ acting on $W_1$ (the space of $\lambda_1$) we find that the trace of the operator $(\lambda_1(y^{-1})\otimes\lambda_1(y))\dot{\lambda}_L(s)$ acting on $W_1\otimes W_1$ is given by  
\[
{\rm Tr}((\lambda_1(y^{-1})\otimes\lambda_1(y))\,\dot{\lambda}_L(s))={\rm dim } (W_1) (={\rm dim } (\lambda_1))
\] 
  
  \par
 Put  
  \[
  \Sigma'_s=\left\{j\in \Sigma_s: \vartheta_{s}(sj^{-1}s)\neq 0\right\}.
  \]
Evaluating both sides of (\ref{eqn-w0}) at $s$, then multiplying by the permutation operator $\dot{\lambda}_L({\sf w}_0)$ on the left and then taking the trace, we obtain 
   
  \begin{equation}\label{eqn-d}
  {\rm vol}(J')^{-1}c_s^2v^2_{s}\omega_{\pi}(-1){\rm dim} (\lambda_1)|\Sigma'_s|={\rm vol}(J')^{-1}c_s(q_{\a}-1)v_{s}{\rm dim}(\lambda_1)^2
  \end{equation}
  which yields $c_s=\omega_{\pi}(-1)\frac{(q_{\a}-1){\rm dim} (\lambda_1)}{|\Sigma'_s|v_{s}}$. On the other hand, evaluating both sides of (\ref{eqn-w0}) at the identity element gives 
  \[
  {\rm vol}(J')^{-2}c_s^2v^2_{s}{\rm vol}(J'sJ')={\rm vol}(J')^{-1}q_{\a}.
  \]
  Since ${\rm vol}({\mathcal I}'s{\mathcal I}')=q_{\a}$, it follows from this that $c_s^2=1$. Now, from (\ref{eqn-d}) we obtain  
  \[
  c_s=\omega_{\pi}(-1) \mbox{ and that } {\rm dim}(\lambda_1)=|\Sigma'_s|v_{s}(q_{\a}-1)^{-1}.
  \]
  \end{proof}
  
 \begin{remark}
 The dimension formula 
 \[
 {\rm dim}(\lambda_1)=|\Sigma'_{{\sf w}_0}|v_{{\sf w}_0}(q_{\a}-1)^{-1}
 \]
  generalizes the case of level-zero representations established by R. Howe \cite[Appendix 3]{Howe85}.
\end{remark}

  Next, in $\widetilde{\sf W}$ we have the relation $t s_{0}=a=s_1t$. It follows from (\ref{length}) that 
    \[
  \vartheta'_t\star \vartheta'_{s_{0}}=\vartheta'_{a}=\vartheta'_{s_1}\star \vartheta'_{t}
  \] 
  in the Hecke algebra ${\mathcal H}'$. We now show that the analogue of this holds in ${\mathcal H}$:
  
  \begin{lemma} \label{lemma-s2}
  Keeping the hypothesis of Lemma~\ref{lemma-s1}, for elements $s_{0}, t$ and $a$ in $\widetilde{\sf W}$ as above, we have
$
\vartheta_t \star \vartheta_{s_{0}}=\vartheta_a= \vartheta_{s_1} \star \vartheta_{t}.
$
As a result
\[
\Psi(\vartheta'_s)=\omega_{\pi}(-1)\vartheta_s
\text{\;
for $s=s_{0},t$. }
\]
\end{lemma}

\begin{proof}
The function $\vartheta_t\star \vartheta_{s_{0}}$ is supported on $J'tJ's_0J'$. It is a direct calculation to see that $t(J'\cap \overline{N})t^{-1}\subset J'\cap N$ and $s_{0}(J'\cap N)s_{0}\subset J\cap \overline{N}$ and therefore $J'tJ's_0J'=J'ts_{0}J'$. Similarly $\vartheta_{s_{1}}\star \vartheta_t$ is supported on $J's_1J'tJ'=J's_1tJ'$.  
 Let us evaluate $\vartheta_t\star \vartheta_{s_0}$ at $ts_0=a$:
\begin{equation}
\begin{split}
 (\vartheta_{t} \star \vartheta_{s_{0}})(t s_{0})
 &=\int\limits_{J'tJ'} \vartheta_{t} (u) \vartheta_{s_{0}}(u^{-1} t s_{0})du
 =\sum_{j \in \Sigma_t} \int_{J^{\prime}}\vartheta_{t}(j t z) \vartheta_{s_{0}}(z^{-1} t^{-1} j^{-1} t s_{0}) dz\\
 &=v_t \sum_{j \in \Sigma_t} 
 \dot{\lambda}_L(t) \vartheta_{s_{0}}(t^{-1} j^{-1} t s_{0}).
\end{split}
\end{equation}
Using the Iwahori factorization of $J'$ one checks that 
\begin{itemize}
\item $J'\cap tJ't^{-1}=J'\cap {\overline{N}}\cdot J'_{L}\cdot t(J'\cap {\overline{N}})t^{-1}$ 
\item $J'\cap s_{0}J's_{0}^{-1}=s_{0}(J'\cap N)s_{0}^{-1}\cdot J'_{L}\cdot (J'\cap N)$ 
\end{itemize} 
Hence we may identify 
\[
\Sigma_t
= \begin{pmatrix} I_n & \varpi_{\mathfrak{k}}^{-1}\mathfrak{H}^1(\beta,\mathfrak{a}) \big\slash  \mathfrak{J}(\beta,\mathfrak{a}) \\ & I_n \end{pmatrix}; \quad
   \Sigma_{s_0}= \begin{pmatrix} I_n &  \\  \varpi_{\mathfrak{k}}\mathfrak{J}(\beta,\mathfrak{a}) \big\slash  \varpi_{\mathfrak{k}}\mathfrak{H}^1(\beta,\mathfrak{a})  & I_n \end{pmatrix}.
\]
Now $\vartheta_{s_{0}}(t^{-1} j^{-1} t s_{0}) \neq 0\Leftrightarrow t^{-1} j^{-1} t s_{0} \in J^{\prime} s_{0} J^{\prime}$. Writing $j=n(x), x\in \varpi_{\k}^{-1}{\mathfrak H}^{1}(\beta,\a)$, we see that 
\[
t^{-1} n(-x) t s_{0}\in J^{\prime}s_{0}J^{\prime}\Leftrightarrow t^{-1}n(-x)ts_{0}\overline{n}(y)s_{0}\in J'
\]
for some $y\in \varpi_{\k}{\mathfrak J}(\beta,\a)$. Since $t^{-1}n(-x)t=\overline{n}(-x\varpi_{\k})$ and $s_{0}\overline{n}(y)s_0=n(\varpi_{\k}^{-1}y\varpi^{-1}_{\k})$, the latter condition holds if and only if $x\in {\mathfrak J}(\beta,\a)$. Consequently
\[
(\vartheta_t\star \vartheta_{s_0})(ts_0)={\rm vol}(J')^{-1}v_t v_{s_{0}}\dot{\lambda}_L(ts_{0})={\rm vol}(J')^{-1}v_tv_{s_{0}}\dot{\lambda}_L(a).
\]
We claim that $v_tv_{s_{0}}=v_{a}$. To see this note that 
    \[
    \begin{split}
   \mathrm{vol} _{J'}(J^{\prime} t   J^{\prime}   ) \mathrm{vol}_{J'}  (J^{\prime} s_{0}    J^{\prime}   )&=\mathrm{vol} (\mathfrak{H}^1(\beta,\mathfrak{a}) \slash \varpi_{\mathfrak{k}} \mathfrak{J}(\beta,\mathfrak{a}))\mathrm{vol} (   \varpi_{\mathfrak{k}}\mathfrak{J}(\beta,\mathfrak{a}) \slash  \varpi_{\mathfrak{k}}\mathfrak{H}^1(\beta,\mathfrak{a})    )\\
   &=[   \mathfrak{H}^1(\beta,\mathfrak{a}):   \varpi_{\mathfrak{k}}\mathfrak{H}^1(\beta,\mathfrak{a})     ]
      \end{split}
       \]
      which equals 
       $
       \mathrm{vol}_{J'}  \left(J^{\prime} a   J^{\prime}   \right)=[ \varpi^{-1}_{\mathfrak{k}}    \mathfrak{H}^1(\beta,\mathfrak{a}): \mathfrak{H}^1(\beta,\mathfrak{a})   )  ].
       $
        Thus $(\vartheta_{t} \star \vartheta_{s_{0}})(t s_{0})=\vartheta_a(a)$ proving \[
        \vartheta_t\star \vartheta_{s_{0}}=\vartheta_a.
        \]
 
 \par
   Next, we prove $(\vartheta_{s_1}\star \vartheta_t)(s_1t)=\vartheta_a(a)$: 
   \[
   \begin{split}
  ( \vartheta_{s_1} \star \vartheta_{t}) (s_1 t)
   &= \int\limits_{J's_1J'} \vartheta_{s_1}(u)  \vartheta_{t}(u^{-1} s_1 t)du
  =  \sum_{j \in \Sigma_{s_1}} \int\limits_{J^{\prime}} \vartheta_{s_1}(js_1z)\vartheta_{t}(z^{-1}s_1j^{-1} s_1  t )dz\\
  & = v_{s_1}   \sum_{j \in \Sigma_{s_1}} 
   \dot{\lambda}_L(s_1) \vartheta_{t}(s_1j^{-1} s_1 t ),
   \end{split}
   \]
   where $\Sigma_{s_1}
     =  \begin{pmatrix} I_n & \varpi_{\mathfrak{k}}^{-1}\mathfrak{H}^1(\beta,\mathfrak{a}) \big\slash  \varpi_{\mathfrak{k}}\mathfrak{J}(\beta,\mathfrak{a}) \\ & I_n \end{pmatrix} $ as in the proof of Lemma~\ref{lemma-s1}. As above, write $j=n(x), x\in \varpi_{\k}^{-1}{\mathfrak H}^{1}(\beta,\a)$, then 
     \begin{eqnarray}
     s_1n(-x) s_1 t \in J^{\prime} t J^{\prime}&\Leftrightarrow & t^{-1}n(y)s_1n(-x)s_1t\in J'\mbox{ for some }y\in \Sigma_t;\nonumber\\&\Leftrightarrow& t^{-1}n(y)tn(-\varpi^{-1}_{\k}x)\in J'\nonumber\\&\Leftrightarrow & x\in {\mathfrak H}^{1}(\beta,\a).\nonumber
     \end{eqnarray}
    Consequently 
         \begin{equation}\label{eqn-s1}
     \begin{split}
     ( \vartheta_{s_1} \star \vartheta_{t})(s_1 t) 
      &={\rm vol}(J')^{-1}
      v_{s_1}v_{t}
      \sum_{x \in \mathfrak{H}^1(\beta,\mathfrak{a}) \slash  \varpi_{\mathfrak{k}}\mathfrak{J}(\beta,\mathfrak{a})} \dot{\lambda}_L(s_1) \dot{\lambda}_L(t) 
      \\
      &=
     {\rm vol}(J')^{-1} v_{s_1}v_{t} \mathrm{vol}(\mathfrak{H}^1(\beta,\mathfrak{a}) \slash  \varpi_{\mathfrak{k}}\mathfrak{J}(\beta,\mathfrak{a}))
       \dot{\lambda}_L(a).
      \end{split}
     \end{equation}
     But 
      \[
      \begin{split}
     \frac{\mathrm{vol}_{J'}  (J^{\prime} s_1    J^{\prime}   )^{1/2} \mathrm{vol} _{J'}(J^{\prime} t    J^{\prime}   )^{1/2}}{\mathrm{vol}(\mathfrak{H}^1(\beta,\mathfrak{a}) \slash  \varpi_{\mathfrak{k}}\mathfrak{J}(\beta,\mathfrak{a}))}
      & =\frac{\mathrm{vol} (\varpi_{\mathfrak{k}}^{-1}\mathfrak{H}^1(\beta,\mathfrak{a}) \slash  \varpi_{\mathfrak{k}}\mathfrak{J}(\beta,\mathfrak{a}))^{1/2} 
      \mathrm{vol} (\mathfrak{H}^1(\beta,\mathfrak{a}) \slash \varpi_{\mathfrak{k}} \mathfrak{J}(\beta,\mathfrak{a}))^{1/2} }{\mathrm{vol}(\mathfrak{H}^1(\beta,\mathfrak{a}) \slash  \varpi_{\mathfrak{k}}\mathfrak{J}(\beta,\mathfrak{a}))} \\
   &  = [\varpi_{\mathfrak{k}}^{-1}\mathfrak{H}^1(\beta,\mathfrak{a}) : \mathfrak{H}^1(\beta,\mathfrak{a})]^{1/2}
      \end{split}
      \]
and hence the product of the volume factors in the last equality in (\ref{eqn-s1}) equals $v_a$ proving 
      \[
       \vartheta_{s_1} \star \vartheta_{t} =\vartheta_a.
\]

 To conclude, since $\Psi$ is support preserving, for $s=s_{0},t$, we have $\Psi(\vartheta'_s)=c_s \vartheta_s$ for some $c_s\in {\mathbb C}^{\times}$. Applying $\Psi$ to the relation $\vartheta'_a
 = \vartheta'_{s_1}\star \vartheta'_{t}$, it follows from Lemma~\ref{lemma-s1}, (\ref{eqn-a}) and the above relation that
 \[
 c_t \omega_{\pi}(-1)(\vartheta_{s_1}\star \vartheta_{t})=\vartheta_a \Rightarrow c_t=\omega_{\pi}(-1).
 \]
 Likewise, using $\vartheta_t\star \vartheta_{s_0}=\vartheta_a$, it follows that $c_{s_0}=\omega_{\pi}(-1)$. 
   \end{proof}

\begin{lemma}\label{lemma-alpha}
For elements $s_0$ and $s_1$ in $\widetilde{\sf W}$ as above, let $\alpha^{\vee}$ be the simple coroot corresponding to $s_1={\sf w}_{0}$. Then 
\[
  \vartheta_{s_1} \star \vartheta_{s_0}=\vartheta_{\alpha^{\vee}}.
  \]
\end{lemma}

\begin{proof}
Since $s_1(J' \cap \overline{N})s_1 \subset J' \cap N$ and $s_0(J' \cap N) s_0 \subset J' \cap \overline{N}$, the function $\vartheta_{s_1} \star \vartheta_{s_0}$ is supported on $J' s_1 J'  s_0 J' =J'  s_1s_0 J'$. Now we evaluate at $s_1s_0=\alpha^{\vee}(\varpi_{\mathfrak k})$:
\[
\begin{split}
  \vartheta_{s_1} \star \vartheta_{s_0}(s_1s_0)&=\int\limits_{J's_1J'} \vartheta_{s_1} (u) \vartheta_{s_{0}}(u^{-1} s_1 s_{0})du
 =\sum_{j \in \Sigma_{s_1}} \int_{J^{\prime}}\vartheta_{s_1}(j s_1 z) \vartheta_{s_{0}}(z^{-1} s_1 j^{-1} s_1 s_{0}) dz\\
 &=v_{s_1} \sum_{j \in \Sigma_{s_1}} 
 \dot{\lambda}_L(s_1) \vartheta_{s_{0}}(s_1 j^{-1} s_1 s_{0}).
\end{split}
\]
As in Lemma~\ref{lemma-s2}, writing $j=n(x)$, $x \in \varpi_{\mathfrak{k}}^{-1}\mathfrak{H}^1(\beta,\mathfrak{a})$, we have 
\[
\begin{split}
  s_1 n(-x) s_1 s_{0} \in J's_0J'  &\Leftrightarrow s_0 \overline{n}(-y) s_1 n(-x) s_1 s_{0} \in J' \mbox{ for some } y \in \varpi_{\mathfrak{k}}\mathfrak{J}(\beta,\mathfrak{a}) \\
  &\Leftrightarrow n(-\varpi^{-2}_{\mathfrak{k}}(x+y)) \in J^{\prime},
  \end{split}
\]
which implicates $x=y=0$. Combining these together, we have
\[
(\vartheta_{s_1} \star \vartheta_{s_0})(s_1s_0)= { {\rm vol}(J')^{-1}        } v_{s_1} v_{s_0} \dot{\lambda}_L(s_1)  \dot{\lambda}_L(s_0)  
= { {\rm vol}(J')^{-1}        } v_{s_1} v_{s_0} \dot{\lambda}_L(s_1s_0). 
\]
But 
\[
\begin{array}{lll}
 {  \mathrm{vol}_{J'} (J^{\prime} s_1   J^{\prime}   ) \mathrm{vol}_{J'}  (J^{\prime} s_{0}    J^{\prime}   )}
 &=&\mathrm{vol}(
 \varpi_{\mathfrak{k}}^{-1}\mathfrak{H}^1(\beta,\mathfrak{a}) \big\slash  \varpi_{\mathfrak{k}}\mathfrak{J}(\beta,\mathfrak{a})   ) 
 \mathrm{vol} (   \varpi_{\mathfrak{k}}\mathfrak{J}(\beta,\mathfrak{a}) \slash  \varpi_{\mathfrak{k}}\mathfrak{H}^1(\beta,\mathfrak{a})    ),\\
 &=&[ \varpi_{\mathfrak{k}}^{-1}\mathfrak{H}^1(\beta,\mathfrak{a}): \varpi_{\mathfrak{k}}\mathfrak{H}^1(\beta,\mathfrak{a})  ],\\&=& {  \mathrm{vol}_{J'} (J^{\prime} \alpha^{\vee}(\varpi_{\mathfrak{k}})   J^{\prime}   ) }.
 \end{array}
\]
Thus $v_{s_1} v_{s_0}=v_{\alpha^{\vee}}$ and $(\vartheta_{s_1} \star \vartheta_{s_0})(s_1s_0)=\vartheta_{\alpha^{\vee}}(\alpha^{\vee})$ proving the lemma.
\end{proof}

\subsubsection{\bf The unramified principal series representation of $G'=GL_2(\mathfrak k)$}
\label{unramified}
We continue with the Haar measure on $G'$ that assigns ${\mathcal I}'$ unit volume. We have the following decompositions $G^{\prime}=\mathcal{I}^{\prime}\widetilde{\sf W}  \mathcal{I}^{\prime}$ and $K^{\prime}=\mathcal{I}^{\prime} \cup \mathcal{I}^{\prime} {\sf w}_0 \mathcal{I}^{\prime}$. Following Bernstein, for $\mu\in \Lambda'_{+}$, let  $\theta'_{\mu}={q_{\a}^{-\ell(\mu)/2} } \vartheta'_{\mu}$. 
For $\mu\in X_{\ast}$, write $\mu=\mu^{+}-\mu^{-}$ with $\mu^{+},\mu^{-}\in \Lambda'_{+}$, and then define $\theta'_{\mu}=\theta'_{\mu_1}\star\theta'^{-1}_{\mu_2}$.
This is well defined and $\theta'_{\mu}\star \theta'_{\nu}=\theta'_{\mu+\nu}$ for all $\mu,\nu\in X_{\ast}$. The commutation relation between $\vartheta'_{{\sf w}_{0}}$ and $\theta'_{\mu}$ is given by the formula (cf. \cite{Lusz}):  
 \begin{equation}\label{commutant}
 \theta'_{\mu}\star \vartheta'_{{\sf w}_{0}}-\vartheta'_{{\sf w}_{0}}\star\theta'_{{\sf w}_{0}'(\mu)}=(q_{\a}-1)\frac{\theta'_{\mu}-\theta'_{{\sf w}_{0}(\mu)}}{1-\theta'_{-\alpha^{\vee}}}.
  \end{equation}

Let $\mathcal{H}^{\prime}_{K^{\prime}}$ be the finite dimensional subalgebra of $\mathcal{H}^{\prime}$ generated by $\vartheta^{\prime}_{\sf w}$ for ${\sf w}  \in {\sf W}$ and let $\mathcal{H}'_{\rm ab}$ be the commutative subalgebra of $\mathcal{H}^{\prime}$ generated by $\vartheta^{\prime}_{\mu}$, $\mu \in \Lambda'_+$, together with their inverses. Then $\mathcal{H}_{\rm ab}'=\mathbb{C}[A^{\prime} \slash A_0^{\prime}]$ is the group algebra of $A^{\prime} \slash A_0^{\prime}$ consisting of functions $\phi : A^{\prime} \slash A_0^{\prime} \longrightarrow \mathbb{C}$ of finite support with product given by convolution. In particular, $\mathcal{H}'_{\rm ab} \cong \mathbb{C}[t_1^{\pm 1},t_2^{\pm 1}]$, where $t_1,t_2$, are indeterminates. Since ${\sf w}_{0}(\mu)=\mu-\langle\mu,\alpha\rangle \alpha^{\vee}$, the fraction $\frac{\theta'_{\mu}-\theta'_{{\sf w}_{0}(\mu)}}{1-\theta'_{-\alpha^{\vee}}}$ belongs to ${\mathcal H}_{\rm ab}'$. The following result follows from the proof of \cite[Proposition 3.7]{Lusz}. (See also \cite{ChanSavin2}.)
\begin{proposition}\label{prop-basis}
The elements $\vartheta'_{\sf w}\star \theta'_{\mu}$, ${\sf w}\in {\sf W}, \mu\in X_{*}$, form a ${\mathbb C}$-basis of ${\mathcal H}'$. Alternately, the multiplication map 
\[
{\sf m}:\mathcal{H}^{\prime}_{K^{\prime}} \otimes_{\mathbb C} {\mathcal H}_{\rm ab}'\longrightarrow {\mathcal H}'
\]
given by ${\sf m}(\vartheta\otimes \theta)=\vartheta\star\theta$ is a linear isomorphism. 
\end{proposition}

Let ${\mathfrak s}'$ denote the inertial class of the pair $(A',1_{\mathfrak k}\times 1_{\mathfrak k})$ in $G'$ and let ${\mathfrak t}'$ denote the corresponding inertial class in $A'$. Then the category ${\mathfrak R}^{{\mathfrak s}'}(G')$ consists of those representations $(\pi,V)$ generated by their ${\mathcal I}'$-fixed vectors and $V^{\mathcal{I}^{\prime}}$ is naturally a $\mathcal{H}'$-module. The map $V \mapsto V^{\mathcal{I}^{\prime}}$ is an equivalence of categories $\mathfrak{R}^{{\mathfrak s}'}(G^{\prime}) \cong \mathcal{H}'-\mathrm{Mod}$.

\par 
Recall $X(A^{\prime})$ the group of unramified quasi-characters of $A^{\prime}$ which is equipped with the structure of a complex torus whose ring of regular functions is $\mathbb{C}[A^{\prime} \slash A_0^{\prime}]$. It can be identified with $(\mathbb{C} \slash \frac{2\pi i}{\log q_{\a}}\mathbb{Z})^2 \cong (\mathbb{C}^{\times})^2$. For $\chi \in X(A^{\prime})$, let $\iota(\chi):=\iota_{B^{\prime}}^{G^{\prime}}(\chi)$ be the unramified principal series representation of $G^{\prime}$ and let ${\mathcal F}(\chi) (={\mathcal F}_{B'}(\chi))$ denote the space of this representation. The dimension of ${\mathcal F}(\chi)^{K^{\prime}}$ is $1$ and it follows from the Iwasawa decomposition of $G'$ that the function $\phi_{K',\chi}$ given by
\[
 {\phi}_{K',\chi}(uak)=\delta_{B^{\prime}}^{1/2}(a)\chi(a) \quad \text{for} \; u \in U^{\prime}, a \in A^{\prime}\; \text{and} \; k \in K^{\prime},
\]
is a basis for ${\mathcal F}(\chi)^{K^{\prime}}$. On the other hand, the dimension of ${\mathcal F}(\chi)^{\mathcal{I}^{\prime}}$ is $2$ and the functions $\phi_{1,\chi}$ and $\phi_{{\sf w}_0,\chi}$ given by 
\[
  \phi_{1,\chi}\;(uak) =
  \begin{cases}
  \delta_{B^{\prime}}^{1/2}(a)\chi(a), &  k \in \mathcal{I}^{\prime} \\
  0,  & k \notin \mathcal{I}^{\prime}
  \end{cases} \\
  \quad
  \text{and}
  \quad
  \phi_{{\sf w}_0,\chi}(uak) =
   \begin{cases}
  \delta_{B^{\prime}}^{1/2}(a)\chi(a), &  k \in \mathcal{I}^{\prime}{\sf w}_0 \mathcal{I}^{\prime} \\
  0, & k \notin \mathcal{I}^{\prime}{\sf w}_0 \mathcal{I}^{\prime}.
  \end{cases} 
\]
forms a basis of ${\mathcal F}(\chi)^{{\mathcal I}'}$. Since $K'={\mathcal I}'\cup {\mathcal I}'{\sf w}_{0}{\mathcal I}'$, it follows that $\phi_{K',\chi}=\phi_{1,\chi}+\phi_{{\sf w}_{0},\chi}$. 

 The proposition below gives the complete structure of ${\mathcal F}(\chi)^{\mathcal{I}^{\prime}}$. The result is not new and an equivalent formulation in a more general context can be found in \cite{Reeder}. One can also find parallel discussions in \cite{Kim, Prasad}. We include a straightforward proof in our setting for the sake of completeness. 

\begin{proposition}
\label{GL2-structure}
Keeping the above notation, we have the following:
\begin{enumerate}[label=$(\mathrm{\alph*})$]
\item\label{GL2-structure-1} ${\phi}'_{{\sf w}_0,\chi}$ is an eigenvector for $\mathcal{H}'_{\rm ab}$. To be more precise, for $\mu \in \Lambda^{\prime}_+$ we have
\[
  \iota_{B^{\prime}}^{G^{\prime}}(\chi) ( \theta'_{\mu}) {\phi}'_{{\sf w}_0,\chi}=  {\sf w}_0(\chi)(\mu(\varpi_{\mathfrak k})) {\phi}'_{{\sf w}_0,\chi}.
\]
\item\label{GL2-structure-2} Let $\iota' :{\mathcal F}(\chi)^{\mathcal{I}^{\prime}}\rightarrow \mathcal{H}^{\prime}_{K^{\prime}}$ denote the map $\iota' (\phi)=\check{\phi}|_{K'}$, i.e., $\iota'(\phi)(k)=\phi(k^{-1})$. Then $\iota'$ is well defined and is
an $\mathcal{H}^{\prime}_{K^{\prime}}$-module isomorphism for the left action on $\mathcal{H}^{\prime}_{K^{\prime}}$ given by convolution. 

\item\label{GL2-structure-3} As $\mathcal{H}^{\prime}$-modules, we get an isomorphism
\[
  {\mathcal F}(\chi)^{\mathcal{I}^{\prime}} \cong   \mathcal{H}^{\prime}  \otimes_{\mathcal{H}_{\rm ab}'} {\mathbb C}_{{\sf w}_0(\chi)}.
\]
\end{enumerate}
\end{proposition}

\begin{proof}
Part \ref{GL2-structure-1} follows from \cite[Lemma 3.9]{cass1} and the definition of $\theta'_{\mu}$. For \ref{GL2-structure-2}, 
we first check that $\iota(\phi)$ is an element of $\mathcal{H}^{\prime}_{K^{\prime}}$, i.e., that it is a bi ${\mathcal I}'$-invariant function on $K'$. Let $i,i^{\prime} \in \mathcal{I}^{\prime}$ and $k \in K^{\prime}$, since $K'=\mathcal{I}^{\prime} {\sf W}\mathcal{I}^{\prime}$, we may take $k\in {\sf W}=\{1,{\sf w}_{0}\}$. Using the Iwahori factorization, decompose $i^{-1}=i_{U^{\prime}}i_{T^{\prime}}i_{\overline{U}^{\prime}}$, $i_{U^{\prime}} \in U^{\prime}$, $i_{T^{\prime}} \in T^{\prime}$ and $i_{\overline{U}^{\prime}} \in \overline{U}^{\prime}$. Then 
\[
\begin{split}
  \iota' (\phi)(i' k i)&=\phi(i^{-1} k^{-1})=\delta_{B^{\prime}}^{1/2}(i_{T^{\prime}})\chi(i_{T^{\prime}})\phi(i_{\overline{U}^{\prime}}k^{-1})\\
  &=\phi(i_{\overline{U}^{\prime}}k^{-1})=\phi(k^{-1} ki_{\overline{U}^{\prime}}k^{-1})=\phi(k^{-1})= \iota' (\phi)(k).
\end{split}
\]
Here, the last equality follows since $ki_{\overline{U}^{\prime}}k^{-1} \in \mathcal{I}^{\prime}$ for $k\in {\sf W}$. To verify that $\iota'$ is an $\mathcal{H}^{\prime}_{K^{\prime}}$-module homomorphism: Suppose that $\vartheta'_0 \in \mathcal{H}^{\prime}_{K^{\prime}}$ and $\phi\in {\mathcal F}(\chi)^{{\mathcal I}'}$. We have
\[
 (\iota'(\vartheta'_0\cdot \phi))(k) =(\vartheta'_0\cdot\phi)(k^{-1})=\int\limits_{K^{\prime}} \vartheta'_0(x) \phi(k^{-1}x) dx
 =\int\limits_{K'} \vartheta'_0(x) \iota'(\phi) (x^{-1}k)dx=( \vartheta'_0\star \iota' (\phi))(k).
\]
To see that $\iota'$ is an isomorphism, it is enough to check that it takes a basis of ${\mathcal F}(\chi)^{{\mathcal I}'}$ to a basis of $\mathcal{H}^{\prime}_{K^{\prime}}$. This is clear since $\iota'(\phi'_{{\sf w},\chi})=\vartheta'_{\sf w}$ for ${\sf w}\in {\sf W}$. Finally, consider the map from 
\begin{equation}\label{eqn-iso}
\mathcal{H}^{\prime}  \otimes_{\mathcal{H}_{\rm ab}'} {\mathbb C}_{{\sf w}_{0}(\chi)}\rightarrow {\mathcal F}(\chi)^{\mathcal{I}^{\prime}}
\end{equation}
given by $\vartheta' \otimes 1\mapsto \vartheta'\cdot \phi'_{{\sf w}_{0},\chi}$. This is clearly a morphism of ${\mathcal H}'$-modules. To see that the map is an isomorphism, observe that Part \ref{GL2-structure-1} and Proposition~\ref{prop-basis} together imply that $\{\vartheta'_{\sf w}\otimes 1: {\sf w}\in {\sf W}\}$ is a ${\mathbb C}$-basis of ${\mathcal H}^{\prime}\otimes _{{\mathcal A}'}{\mathbb C}_{{\sf w}_{0}(\chi)}$. On the other hand, we also have the basis $\{\phi'_{1,\chi},\phi'_{{\sf w}_{0},\chi}\}$ of ${\mathcal F}(\chi)^{{\mathcal I}'}$ and 
\begin{eqnarray}
\vartheta'_{1}\otimes 1\mapsto \vartheta'_{1}\cdot \phi'_{{\sf w}_{0},\chi}&=&\phi'_{{\sf w}_{0},\chi};\nonumber\\
\vartheta'_{{\sf w}_{0}}\otimes 1\mapsto \vartheta'_{{\sf w}_{0}}\cdot \phi'_{{\sf w}_{0},\chi}&=&c_1\phi'_{1,\chi}+c_2\phi'_{{\sf w}_{0},\chi}.\nonumber 
\end{eqnarray}
To determine the complex numbers $c_1,c_2$, we apply the map $\iota'$ to the second equation and use \ref{GL2-structure-2} to obtain:
\[
\vartheta'_{{\sf w}_{0}}\star \vartheta'_{{\sf w}_{0}}=c_1\vartheta'_{1}+c_2\vartheta'_{{\sf w}_0}
\]
from which it follows that $c_1=q_{\a}, c_2=q_{\a}-1$. Hence the change of basis matrix with respect to these bases is invertible establishing that the map (\ref{eqn-iso}) is an isomorphism. 
\end{proof}

\begin{remark}
Evaluating the second equation at ${\sf w}_{0}$, we obtain the volume formula 
\[
{\rm vol}(\overline{B}'{\mathcal I}'\cap {\mathcal I}'{\sf w}_{0}{\mathcal I}')=q_{\a}-1.
\] 
\end{remark}
Henceforth we assume that $\chi \neq  {\sf w}_0(\chi)$ (such a $\chi$ is also called {\it regular}).
Consider the intertwining operator $A(\chi,{\sf w}_{0}) : {\mathcal F}(\chi) \longrightarrow {\mathcal F}({\sf w}_{0}(\chi))$ defined as in (\ref{eqn-int}) for the data $(G',A',\chi)$, i.e., 
\[
 A(\chi,{\sf w}_{0})(\phi)(g)=\int\limits_{U^{\prime}} \phi({\sf w}_0ug)\;du, \quad \phi \in  {\iota}(\chi), \; g \in G^{\prime}.
\]
It defines a regular function on the domain of regular characters. Since $\chi$ is regular, every $G^{\prime}$-morphism from ${\mathcal F}(\chi)$ to ${\mathcal F}({\sf w}_{0}(\chi))$ is a scalar multiple of $A(\chi,{\sf w}_{0})$. The operator $A(\chi,{\sf w}_{0})$ is in particular a $K^{\prime}$-homomorphism, therefore it takes $\phi'_{K^{\prime},\chi}$ to a scalar multiple of $\phi'_{K^{\prime},{\sf w}_{0}(\chi)}$. This scalar has been computed by Casselman for a general unramified principal series representation in \cite{cass1}. In the current setting, his formula takes the following shape:

\begin{lemma}
\label{intertwining-spherical}
Suppose $\chi=\chi_1 \otimes \chi_2$ in $X(A')$ is regular. Then 
\[
  A(\chi,{\sf w}_{0})(\phi'_{1,{\chi}})=(c_{{\sf w}_0}(\chi)-1)\phi'_{1,\chi}+\frac{1}{q_{\a}}\phi'_{{\sf w}_0,\chi}  
  \quad \text{and} \quad
 A(\chi,{\sf w}_{0})(\phi'_{{\sf w}_{0},\chi})=\phi'_{1,\chi}+\left(c_{{\sf w}_0}(\chi)-\frac{1}{q_{\a}} \right) \phi'_{{\sf w}_{0},\chi}, 
\]
where  
$
c_{{\sf w}_0}(\chi)=\dfrac{1-q^{-1}_{\a}\chi_1(\varpi_{\mathfrak k})\chi_2^{-1}(\varpi_{\mathfrak k})}{1-\chi_1(\varpi_{\mathfrak k})\chi_2^{-1}(\varpi_{\mathfrak k})}.
$
Consequently 
\[
  A(\chi,{\sf w}_{0})(\phi'_{K',\chi})= c_{{\sf w}_0}(\chi) \phi'_{K',{\sf w}_{0}(\chi)}.
\]
\end{lemma}

\par
Fix an additive character $\psi'$ of ${\mathfrak k}$ that is unramified and let us look at the corresponding local coefficient $C_{\psi'}(\chi)$. Let ${\iota}(\chi)_{{\sf w}_0}$ denote $B^{\prime}$-submodule of functions in ${\iota}(\chi)$ supported on the big cell $B^{\prime} {\sf w}_0 U^{\prime}$. As in \S\ref{LS}, there is a unique non-zero Whitaker functional $\Omega'_{\chi}$ defined on $\iota(\chi)$ so that, for $\phi\in {\iota}(\chi)_{{\sf w}_{0}}$, 
\[
\Omega'_{\chi}(\phi)=\int\limits_{U^{\prime}} \phi({\sf w}_0u)\psi^{-1}(u)du.
\]
This formula holds on all of $\iota(\chi)$ as a principal value integral. By definition of the local coefficient $C_{\psi'}(\chi)$, we have 
\begin{equation}\label{lcprime}
  C_{\psi'}(\chi)(\Omega^{\prime}_{{\sf w}_{0}(\chi)} \circ A(\chi,{\sf w}_{0}))=\Omega^{\prime}_{\chi}.
\end{equation} 
As mentioned in the beginning of Section \ref{sec:equal}, the explicit form of $C_{\psi'}(\chi)$ is known for a general unramified principal series representation by the work of Casselman and Shalika \cite{cass2}. In our setting, their result is as follows. 
\begin{lemma}
\label{GL(2)-local-coefficient}
For any $\chi=\chi_1 \otimes \chi_2\in X(A')$, we have 
\[
  \Omega_{\chi}^{\prime}(\phi'_{1,\chi})=-q_{\a}^{-1}\chi_1(\varpi_{\mathfrak k})\chi_2^{-1}(\varpi_{\mathfrak k}) \quad \text{and} \quad
  \Omega_{\chi}^{\prime}(\phi'_{{\sf w}_0,\chi})=1.
\]
If in addition $\chi$ is regular, by evaluating both sides of the equation (\ref{lcprime}) at the spherical function $\phi'_{K^{\prime},\chi}=\phi'_{1,\chi}+\phi'_{{\sf w}_0,\chi}$, we obtain 
\[
  C_{\psi'}(\chi)=\frac{1-\chi_1(\varpi_{\mathfrak k})\chi_2^{-1}(\varpi_{\mathfrak k})}{1-q_{\a}^{-1}\chi_1^{-1}(\varpi_{\mathfrak k})\chi_2(\varpi_{\mathfrak k})}.
\]
\end{lemma}

\subsubsection{\bf Computation of the local coefficient for non-split cases} 
\label{computation}

In this subsection, we transfer the above results from $G'$ to $G$ using the support preserving algebra isomorphism 
\[
{\mathcal H}^{\prime} \xrightarrow{\Psi} {\mathcal H}.
\]
Recall the basis elements $\{\vartheta'_{\sf w}\}$ and $\{\vartheta_{\sf w}\}$ in ${\mathcal H}'$ and ${\mathcal H}$, respectively. As in the case of $G'$, define $\theta_{\mu}:=q_{\a}^{-\ell(\mu)/2}\vartheta_{\mu}$, $\mu \in \Lambda_{+}^{\prime}$. For an arbitrary $\mu\in X_{\ast}$, write $\mu=\mu_1-\mu_2, \mu_1,\mu_2\in \Lambda'_{+}$, and set $\theta_{\mu}=\theta_{\mu_1}\star \theta_{\mu_2}^{-1}$. 

\begin{lemma}\label{lemma-abelian}
Keep the above notation and assume that the Haar measures on $G'$ is normalized so that ${\mathcal I}'$ has unit volume. Then the the following relations hold:
\begin{align*}
  \theta_{\mu_1} \star \theta_{\mu_2}&=\theta_{\mu_1+\mu_2} \mbox{ for } \mu_1, \mu_2 \in X_{\ast};\nonumber \\
  \vartheta_{{\sf w}_0} \star \vartheta_{{\sf w}_0}&=q_{\a}\vartheta_1+\omega_{\pi}(-1)(q_{\a}-1)\vartheta_{{\sf w}_0}.\nonumber 
 \end{align*}
 \end{lemma}
 \begin{proof}
 The first relation follows since $J'\varpi_E^{\mu}J'\varpi_E^{\mu'}J'=J'\varpi_E^{\mu+\mu'}J'$, which in turn follows from the Iwahori decomposition $J'=J'\cap \overline{N}\cdot J_L\cdot J'\cap N$ and (\ref{J-prime}). The second assertion follows from Lemma~\ref{lemma-s1} and applying $\Psi$ to the quadratic relation (\ref{quad-prime}).    \end{proof}

Let $\mathcal{H}_{\rm ab}$ be the abelian subalgebra of $\mathcal{H}$ generated by $\vartheta_{\mu}$ for $\mu \in \Lambda_+^{\prime}$ and their inverses. For $\mu\in X_{\ast}'$, write $\Psi(\vartheta'_{\mu})=\xi_{\mu}\vartheta_{\mu}$ for $\xi_{\mu}\in {\mathbb C}^{\times}$. Lemma~\ref{lemma-abelian} implies that $\mu\mapsto \xi_{\mu}$ defines a character of $X_{\ast}$, or equivalently, an unramified character of $A'$. We denote this character as $\xi$. 
\begin{lemma}\label{lemma-trivial}
$\xi$ is the trivial character of $A'$.
\end{lemma}

\begin{proof}
Recall we have elements $a,s_0,s_1\in \widetilde{\sf W}$ and note that $a$ corresponds to the dominant weight $(1,0)$. Since $\Psi(\vartheta_a')=\vartheta_a$ by definition, it follows that (i) $\Psi(\vartheta'_{a^m})=\vartheta_{a^m}$ for all $m\in \mathbb Z$. On the other hand,
\[
\Psi(\vartheta'_{\alpha^{\vee}})=\Psi(\vartheta'_{s_1}\star\vartheta'_{s_{0}})=\vartheta_{s_{1}}\star\vartheta_{s_0}=\vartheta_{\alpha^{\vee}}.
\]
Here, the second equality uses Lemma~\ref{lemma-s1} and Lemma~\ref{lemma-s2} and the third equality follows from Lemma~\ref{lemma-alpha}. Note $\alpha^{\vee}$ corresponds to the weight $(1,-1)$, put $\mu=(1,0)-(1,-1)$. Then $\vartheta'_{\mu}=\vartheta'_a\star \vartheta'^{-1}_{\alpha^{\vee}}$ and $\vartheta_{\mu}=\vartheta_a\star\vartheta^{-1}_{\alpha^{\vee}}$. It follows that $\Psi(\vartheta'_{\mu})=\vartheta_{\mu}$.  This in turn implies that (ii) $
\Psi(\vartheta'_{b^n})=\vartheta_{b^n}$, $n\in \Z$, where $b=\begin{psmallmatrix}1&\\&\varpi_{\k}\end{psmallmatrix}$. Now, (i) and (ii) combined implies that $\Psi(\vartheta'_{\mu})=\vartheta_{\mu}$ for any $\mu\in X_{\ast}$ and hence the assertion. \end{proof}

Let ${\mathcal H}_{K'}$ be the subalgebra of functions in ${\mathcal H}$ supported on $J'K'J'=J'{\sf W}J'$. Then $\Psi$ restricts to a support preserving isomorphism between subalgebras:
\[
{\mathcal H}'_{K'}\cong {\mathcal H}_{K'}\mbox{ and }{\mathcal H}_{\rm ab}'\cong {\mathcal H}_{\rm ab}. 
\]
By Proposition~\ref{prop-basis}, we may write ${\mathcal H}$ as a ``twisted" tensor product:
\[
  \mathcal{H} \cong  \mathcal{H}_{K'} \tilde{\otimes}_{\mathbb{C}} \mathcal{H}_{\rm ab}.
\]

 Next, for an unramified character $\chi\in X(L)$ of $L$, we turn our attention to the space $\mathcal{F}_P(\sigma \otimes \chi)_{\lambda'}$ of $\lambda'$-invariants which is a ${\mathcal H}$-module (cf. \S\ref{gen}). Note $W=W_1\otimes W_1$ is the space of $\lambda_L$ as well as $\lambda'$. There is a natural isomorphism between  
 \[
 \mathcal{F}_P(\sigma \otimes \chi)_{\lambda'}\otimes_{\mathbb C}W\longrightarrow \mathcal{F}_P(\sigma \otimes \chi)^{\lambda'}
 \]
 given by $\phi\otimes w\mapsto \phi(w)$. By \cite[Lemma 3.2.4]{Kim}, any $f\in \mathcal{F}_P(\sigma \otimes \chi)^{\lambda'}$ is determined by its restriction to $K'$. For $w\in W$, we have the function $\varphi_w\in \text{c-Ind }\tilde{\lambda}_L$ and the corresponding function $f_w=f_{\varphi_w}\in {\mathcal F}_P(\sigma\otimes \chi)^{\lambda'}$ (supported on $PJ'$) which of course depends on $\chi$ as in \S\ref{sec:unequal}. We set 
 \[
 f_{1,w,\chi}={\rm vol}(J')^{-1}f_{\varphi_w}.
 \]
 Using the extended representation $\dot{\lambda}_L$, we may also define a similar function $f=f_{{\sf w}_{0},w,\chi}\in {\mathcal F}_P(\sigma\otimes \chi)^{\lambda'}$ supported on $P{\sf w}_{0}J'$ as follows:
\[
f(p{\sf w}_0j)={\rm vol}(J')^{-1}\chi(p)\delta_{P}^{1/2}(p)\sigma(p)\varphi_{\dot{\lambda}_L({\sf w}_{0})\lambda'(j)w}.
\]
This is well defined since the map $w\mapsto \varphi_w$ is $\tilde{J}_L$-intertwining. For ${\sf w}\in {\sf W}$, the resulting map 
\[
  \phi_{{\sf w},{\chi}} : W \longrightarrow \mathcal{F}_P(\sigma \otimes \chi) \quad \text{given by }w\mapsto f_{{\sf w}, w,\chi} 
\]  
is clearly a $J'$-embedding, i.e., it belongs to ${\mathcal F}(\sigma\otimes\chi)_{\lambda'}$. As explained in \cite[Lemma 3.2.9]{Kim}, the set $\{\phi_{1,{\chi}} ,  \phi_{{\sf w}_0,\chi} \}$ forms a basis of $\mathcal{F}_P(\sigma \otimes \chi)_{\lambda^{\prime}}$. 

If $\phi\in {\mathcal F}(\sigma\otimes\chi)_{\lambda'}$, $w\in W$, and $x\in J'{\sf W}J'$, then $\phi(w)(x)$ belongs to the $\lambda_L$-isotypic component (c-${\rm Ind}\,\tilde{\lambda}_L)^{\lambda_L}\cong W$ and consequently $w\mapsto \phi(w)(x)$ defines an element in $\mathrm{End}_{\mathbb C}(W)$. By \cite[Lemma 3.2.6]{Kim} the map 
\[
\iota_2: \mathcal{F}_P(\sigma \otimes \chi)_{\lambda^{\prime}} \rightarrow\mathcal{H}_{K'}(G,{\check{\lambda}}^{\prime})
\]
given by $\iota_2(\phi)(x)=\phi(w)(x), x\in J'{\sf W}J'$, is $\mathcal{H}_{K'}(G,{\check{\lambda}}^{\prime})$-equivariant. Then composing $\iota_2$ with the anti-isomorphism $h\mapsto \check{h}$ from $\mathcal{H}_{K'}(G,{\check{\lambda}}^{\prime})\rightarrow \mathcal{H}_{K'}$, we obtain the $\mathcal{H}_{K'}$-equivariant map $\iota: \mathcal{F}_P(\sigma \otimes \chi)_{\lambda^{\prime}} \rightarrow\mathcal{H}_{K'}$.

We have the following analogue of Proposition~\ref{GL2-structure} for the group $G$:

\begin{proposition}\cite[Proposition 3.2.10.]{Kim} Keeping the above notation, the following holds:
\label{GLn-structure}
\begin{enumerate}[label=$(\mathrm{\alph*})$]
\item\label{GLn-structure-1}$\phi_{{\sf w}_0,\chi}$ is an eigenvector for $\mathcal{H}_{\rm ab}$. More precisely, for $\mu \in \Lambda_+^{\prime}$, we have
\[
  \iota_P^G(\sigma \otimes \chi) \left( \theta_{\mu}     \right)\phi_{{\sf w}_0,\chi}={{\sf w}_0}(\chi)(\mu(\varpi_{\mathfrak{k}}))\phi_{{\sf w}_{0},\chi}.
\]
\item\label{GLn-structure-2} The map $\iota$ (defined above) is an $\mathcal{H}_{K'}$-module isomorphism. 
\item\label{GLn-structure-3} As $\mathcal{H}$-modules, the map 
\[
  \mathcal{H} \otimes_{\mathcal{H}_{\rm ab}} \mathbb{C}_{{\sf w}_0(\chi)} \simeq \mathcal{F}_P(\sigma \otimes \chi)_{\lambda^{\prime}}
\]
given by $\vartheta\otimes 1\mapsto \vartheta\cdot \phi_{{\sf w}_{0},\chi}$ is an isomorphism. 
\end{enumerate}
\end{proposition}

Observe that $\chi|_{A'}$ is an unramified character of $A'$, we simply write $\chi$ to denote this character. Thus we can also form the induced representation ${\mathcal F}_{B'}(\chi)$. Let 
\[
\Phi_{\chi} : \mathcal{F}_{B'}(\chi)^{\mathcal{I}^{\prime}} \rightarrow \mathcal{F}_P(\sigma \otimes \chi)_{\lambda'}
\]
denote the linear isomorphism extending $\Phi_{\chi}(\phi'_{1,\chi})=\phi_{1,\chi}$ and $\Phi_{\chi}(\phi'_{{\sf w}_0,\chi})=\omega_{\pi}(-1)(v_{{\sf w}_0}\phi_{{\sf w}_0,\chi})$. The following lemma is a reformulation of \cite[Proposition 3.2.11]{Kim}. 

\begin{lemma}
\label{induced-isom}
The diagram below commutes with $\Phi=\Phi_{\chi}$.
\[
\xymatrix@C+=3cc{
\mathcal{H} \times \mathcal{F}_P(\sigma \otimes \chi)_{\lambda^{\prime}}    \ar[r]
&  \mathcal{F}_P(\sigma \otimes \chi)_{\lambda^{\prime}}  \\
   \mathcal{H}^{\prime}  \times {\mathcal F}_{B'}(\chi)^{\mathcal{I}^{\prime}}   \ar[u]^{\Psi \times \Phi_{\chi}} \ar[r] &{\mathcal F}_{B'}(\chi)^{\mathcal{I}^{\prime}} \ar[u]_{\Phi_{\chi}},\\
      }
\]
where the horizontal arrows are given by the module action of $\mathcal{H}$ and $\mathcal{H}^{\prime}$ on $\mathcal{F}_P(\sigma \otimes \chi)_{\lambda^{\prime}} $ and ${\mathcal F}_{B'}(\chi)^{\mathcal{I}^{\prime}}$, respectively. Put differently, $\Phi_{\ast}({\mathcal F}_{P}(\sigma\otimes \chi)_{\lambda'})\cong  {\mathcal F}_{B'}(\chi)^{\mathcal{I}^{\prime}}$ as ${\mathcal H}'$-modules.

\end{lemma}
\begin{proof}
We need to check
\begin{equation}
\label{basis}
\Phi_{\chi}(\vartheta'\cdot \phi')=\Psi(\vartheta')\cdot \Phi_{\chi}(\phi'), \vartheta'\in {\mathcal H}', \phi'\in {\mathcal F}(\chi)^{{\mathcal I}'}.
\end{equation}
Since by part \ref{GL2-structure-3} of Proposition~\ref{GL2-structure}, any $\phi'\in {\mathcal F}_{B'}(\chi)^{{\mathcal I}'}$ can be written as $\phi'=\vartheta'_{0}\cdot \phi'_{{\sf w}_{0},\chi}$ for some $\vartheta'_{0}\in {\mathcal H}'$, it is sufficient to verify \eqref{basis} for $\phi'=\phi'_{{\sf w}_{0},\chi}$. Further,  by Proposition~\ref{prop-basis}, we may reduce the verification to $\vartheta'= \vartheta'_{{\sf w}_{0}}\mbox{ or } \theta'_{\mu}\in {\mathcal A}'$. 

Let us show $\Phi_{\chi}(\vartheta'_{{\sf w}_{0}}\cdot \phi'_{{\sf w}_{0},\chi})=\Psi(\vartheta'_{{\sf w}_{0}})\cdot \Phi_{\chi}(\phi'_{{\sf w}_{0},\chi})$. Since $\vartheta'_{{\sf w}_{0}}\cdot \phi'_{{\sf w}_0,\chi}=q_{\a}\phi'_{1,\chi}+(q_{\a}-1)\phi'_{{\sf w}_0,\chi}$, the left hand side is given by 
\[
\Phi_{\chi}(\vartheta'_{{\sf w}_{0}}\cdot \phi'_{{\sf w}_{0},\chi})=q_{\a}\phi_{1,\chi}+\omega_{\pi}(-1)(q_{\a}-1)v_{{\sf w}_0}\phi_{{\sf w}_{0},\chi}.
\]
On the other hand, by Lemma~\ref{lemma-s1}, the right hand side equals $v_{{\sf w}_{0}}(\vartheta_{{\sf w}_{0}}\cdot \phi_{{\sf w}_{0},\chi})$. Say,
\[
\vartheta_{{\sf w}_{0}}\cdot \phi_{{\sf w}_{0},\chi}=c_1\phi_{1,\chi}+c_2\phi_{{\sf w}_{0},\chi}.
\]
Apply $\iota$ to this equation to get $\vartheta_{{\sf w}_0}\star \iota(\phi_{{\sf w}_{0},\chi})=c_1\iota(\phi_{1,\chi})+c_2\iota(\phi_{{\sf w}_{0},\chi})$. But $ \iota(\phi_{{\sf w}_0,\chi})=v^{-1}_{{\sf w}_{0}}\vartheta_{{\sf w}_0}$ and $\iota(\phi_{1,\chi})=\vartheta_{1}$. Using the quadratic relation satisfied by $\vartheta_{{\sf w}_0}$ we obtain $c_1=q_{\a}v^{-1}_{{\sf w}_{0}}, c_2=(q_{\a}-1)\omega_{\pi}(-1)$, giving us the desired equality. 

That $\Phi_{\chi}(\theta'_{\mu}\cdot \phi'_{{\sf w}_{0},\chi})=\Psi(\theta'_{\mu})\cdot \Phi_{\chi}(\phi'_{{\sf w}_{0},\chi})$ follows from Lemma~\ref{lemma-trivial} and the fact that $\phi'_{{\sf w}_0,\chi}$ and $\phi_{{\sf w}_0,\chi}$ are eigenfunctions for ${\mathcal H}_{\rm ab}'$ and ${\mathcal H}_{\rm ab}$, respectively, with the same eigencharacter ${\sf w}_{0}(\chi)$. 
\end{proof}

Recall the spherical function $\phi'_{K',\chi}\in {\mathcal F}_{B'}(\chi)^{{\mathcal I}'}$ from Lemma~\ref{GL(2)-local-coefficient}, the line generated by this vector is invariant for the action of ${\mathcal H}'_{K'}$. Let us put $\phi_{K',\chi}=\phi_{1,\chi}+\omega_{\pi}(-1)v_{{\sf w}_{0}}\phi_{{\sf w}_{0},\chi}$, then $\Phi(\phi'_{K',\chi})=\phi_{K',\chi}$ and it follows from the above lemma that the line generated by $\phi_{K'}$ is invariant for the action of ${\mathcal H}_{K'}$. Since 
\[
\mathrm{Hom}_G(\mathcal{F}_P(\sigma \otimes \chi),\mathcal{F}_P(\sigma \otimes {\sf w}_{0}(\chi))) \cong \mathrm{Hom}_{\mathcal{H}} (\mathcal{F}_P(\sigma \otimes \chi)_{\lambda^{\prime}},\mathcal{F}_P(\sigma \otimes {\sf w}_{0}(\chi))_{\lambda^{\prime}}),
\]
the intertwining operator $A(\chi,\sigma,{\sf w}_0)$ (see (\ref{eqn-int})) induces a natural map of ${\mathcal H}$-modules from $ \mathcal{F}_P(\sigma \otimes \chi)_{\lambda^{\prime}} \longrightarrow \mathcal{F}_P(\sigma \otimes {\sf w}_{0}(\chi))_{\lambda^{\prime}}$ given by composition. Suppose $\chi|_{A'}$ is regular so that the representations $\iota_P^G(\sigma \otimes \chi)$, $\iota_P^G(\sigma \otimes {\sf w}_{0}(\chi))$, $\iota_{B'}^{G'}(\chi)$ and $\iota_{B'}^{G'}({\sf w}_0(\chi)$ are all irreducible. In this circumstance, we have the following:

\begin{proposition}
\label{GL(n)-intertwining}
Suppose $\chi=\chi_1\otimes\chi_2\in X(L)$ is regular in the above sense. Then
\[
  A(\chi,\sigma,{\sf w}_0)\circ \phi_{K',\chi}=\mathrm{vol}( J^{\prime} \cap N) v_{{\sf w}_0}\omega_{\pi}(-1) \frac{1-q^{-1}_{\mathfrak{a}}\chi_1(\varpi_{\k})\chi_2^{-1}(\varpi_{\k})}{1-\chi_1(\varpi_{\k})\chi_2^{-1}(\varpi_{k})} \phi_{K',{\sf w}_{0}(\chi)}.
\]
\end{proposition}

\begin{proof}

From Lemma \ref{induced-isom}, it is clear that $\Phi=\Phi_{\chi}$ induces the equivalence
\[
\Phi_{\ast}:
\mathrm{Hom}_{\mathcal{H}^{\prime}}({\mathcal F}_{B'}(\chi)^{\mathcal{I}^{\prime}},{\mathcal F}_{B'}({\sf w}_{0}(\chi))^{\mathcal{I}^{\prime}})
\overset{\cong}{\longrightarrow}
\mathrm{Hom}_{\mathcal{H}} (\mathcal{F}_P(\sigma \otimes \chi)_{\lambda^{\prime}},\mathcal{F}_P(\sigma \otimes {\sf w}_{0}(\chi))_{\lambda^{\prime}}),
\]
where $\Phi_{\ast}(A)= \Phi_{{\sf w}_{0}(\chi)}\circ A\circ \Phi_{\chi}^{-1}$. By irreducibility, there exists a constant $c(\chi,\sigma) \in \mathbb{C}^{\times}$ so that 
\begin{equation}\label{eqn-Aint}
A(\chi,\sigma,{\sf w}_0)=c(\chi,\sigma)\Phi_{\ast}(A(\chi,{\sf w}_{0})).
\end{equation}
Evaluating both sides of (\ref{eqn-Aint}) at $\phi_{K',\chi}$, it follows from Lemma \ref{intertwining-spherical} and Lemma \ref{induced-isom} that
\[
 A(\chi,\sigma,{\sf w}_0)\circ \phi_{K',\chi}=c(\chi,\sigma)\Phi_{\ast}(A(\chi,{\sf w}_0))\circ\phi_{K',\chi}
 =c(\chi,\sigma)c_{{\sf w}_0}(\chi)\phi_{K',{\sf w}_{0}(\chi)}.
\]
On the other hand, if we evaluate both sides of (\ref{eqn-Aint}) at $v_{{\sf w}_0}\omega_{\pi}(-1)\phi_{{\sf w}_0,\chi}$, we obtain 
\[
\begin{split}
v_{{\sf w}_0}\omega_{\pi}(-1)A(\chi,\sigma,{\sf w}_0)\circ \phi_{{\sf w}_0,\chi}
&=c(\chi,\sigma)\Phi_{\ast}(A(\chi,{\sf w}_0))\circ (v_{{\sf w}_0}\omega_{\pi}(-1)  \phi_{{\sf w}_0,\chi})\\
&=c(\chi,\sigma)(\phi_{1,\chi}+(c_{{\sf w}_0}(\chi)-q_{\a}^{-1})v_{{\sf w}_0}\omega_{\pi}(-1)\phi_{{\sf w}_0,\chi}).
\end{split}
\]
Now, apply both sides of the second equation to an arbitrary $w\in W$ and then evaluate at $1$ to get 
\begin{equation}
\label{identity-evaluation}
\begin{split}
{\rm vol}(J')^{-1} c(\chi,\sigma)\varphi_w&=v_{{\sf w}_0}\omega_{\pi}(-1)A(\chi,\sigma,{\sf w}_0)(f_{{\sf w}_0,w,\chi})(1)
 =v_{{\sf w}_0}\omega_{\pi}(-1)\int\limits_{N} f_{{\sf w}_0,w,\chi}({\sf w}_0^{-1}n) dn\\
 &=v_{{\sf w}_0}\omega_{\pi}(-1)\int\limits_{J^{\prime} \cap N} f_{{\sf w}_0,w,\chi}({\sf w}_0^{-1}n) dn
 ={\rm vol}(J')^{-1}v_{{\sf w}_0}\omega_{\pi}(-1) \mathrm{vol}(J^{\prime}\cap N)\varphi_{\dot{\lambda}_L({\sf w}_0)w}.\\
\end{split} 
\end{equation}
Specializing to a symmetric tensor $w$, meaning of the form $w=w_1\otimes w_1$, we conclude 
\begin{equation}\label{c-value}
c(\chi,\sigma)=v_{{\sf w}_0}\omega_{\pi}(-1) \mathrm{vol}(J^{\prime}\cap N).
\end{equation}
(Note that the constant $c(\chi,\sigma)$ is independent of $\chi$). 
\end{proof}

Let us return to the local coefficient defined in \S\ref{LS}. Recall the induced representation $\mathrm{Ind}_U^G(\overline{\psi})$ from \S\ref{whitt}. Let c-$\mathrm{Ind}_U^G(\overline{\psi})$ denote the sub-representation consisting of functions that are compactly supported modulo $U$. Let \textit{sgn} denote the one-dimensional representation of $\mathcal{H}^{\prime}_{K^{\prime}}$ on $\mathbb{C}$ wherein  $\vartheta^{\prime}_{{\sf w}}$ acts as $(-1)^{\ell({\sf w})}$, ${\sf w}\in {\sf W}$. By \cite[Corollary 4.4]{ChanSavin2} and \cite[Theorem 3.4]{ChanSavin1} applied for $\check{\lambda}'$, we have 
\[
 \mathrm{c}\text{-}\mathrm{Ind}_{U^{\prime}}^{G^{\prime}}(\overline{\psi})^{\mathcal{I}^{\prime}} \cong \mathcal{H}'\otimes_{\mathcal{H}'_{K'}} \textit{sgn}\cong  \mathrm{c}\text{-}\mathrm{Ind}_U^G(\overline{\psi})_{\check{\lambda}'}
\]
as ${\mathcal H}(G',1_{{\mathcal I}'})\cong {\mathcal H}(G,\check{\lambda}')$-modules. Thus 
\begin{equation}\label{Gelfand-Graev}
\mathrm{c}\text{-}\mathrm{Ind}_{U^{\prime}}^{G^{\prime}}(\overline{\psi})^{\mathcal{I}^{\prime}}\cong \mathrm{c}\text{-}\mathrm{Ind}_U^G(\overline{\psi})_{\check{\lambda}'}
\end{equation}
as ${\mathcal H}(G',1_{{\mathcal I}'})\cong {\mathcal H}(G,\check{\lambda}')$-modules.

Let $\ast$ denote the linear dual.  We can linearly dualize (\ref{Gelfand-Graev}) to obtain the following isomorphism $\Xi$  of dual spaces as $\mathcal{H}^{\prime} \cong \mathcal{H}$-modules:
\[
\begin{split}
\Xi: \mathrm{Ind}_{U^{\prime}}^{G^{\prime}}(\psi)^{\mathcal{I}^{\prime}} \overset{(1)}{\cong}
 \left({\mathrm{c}\text{-}\mathrm{Ind}_{U'}^{G'}(\overline{\psi})}^{\spcheck}\right)^{\mathcal{I}^{\prime}}
&\overset{(2)}{\cong}
 \left(\mathrm{c}\text{-}\mathrm{Ind}_{U^{\prime}}^{G^{\prime}}(\overline{\psi})^{\mathcal{I}^{\prime}}\right)^{\ast} \\
 & \cong
\left( \mathrm{c}\text{-}\mathrm{Ind}_U^G(\overline{\psi})_{{\check{\lambda}}'} \right)^{\ast}
\overset{(3)} {\cong}
  \left( \mathrm{c}\text{-}\mathrm{Ind}_U^G(\overline{\psi})^{\spcheck} \right)_{\check{\lambda}'}
\overset{(1)}{\cong} \mathrm{Ind}_U^G(\psi)_{\lambda^{\prime}}.
\end{split}
\]
 Here, (1) is given by the duality theorem \cite[\S3.5]{BH}; (2) and (3) follow from the fact that for any smooth representation $(\pi,V)$ of $G$ and a compact open subgroup $K\leq G$ one has $(V^{\spcheck})^K\cong (V^K)^{\ast}$ (cf. \cite[\S3.1]{ChanSavin1}). 
  
Recall the Whittaker functional $\Omega_{\chi}^{\prime}$ in equation (\ref{lcprime}) which is defined with respect to a level zero additive character. We need to shift this to get a functional with respect to $\psi$ which is of level one. To that end, we define $\tilde{\Omega}'_{\chi}(f')=\Omega'_{\chi}(R(\begin{smallmatrix}\varpi_{\k}&\\&1\end{smallmatrix})f')$, $f'\in {\iota}_{B'}^{G'}(\chi)$. Let $\omega'_{\chi} : {\mathcal F}_{B'}(\chi) \rightarrow \mathrm{Ind}_{U^{\prime}}^{G^{\prime}}(\psi)$ be the ``Whittaker map" corresponding to $\tilde{\Omega}'_{\chi}$ determined by Frobenius reciprocity. We can transfer $\omega'_{\chi}$ via the isomorphisms $\Phi_{\chi}$ and $\Xi$ to get the map $(\omega'_{\chi})_G$ of ${\mathcal H}$-modules: 
\begin{equation}
\label{local-coefficient-commute}
\xymatrix@C+=5cc{
\mathcal{F}_P(\sigma \otimes \chi)_{\lambda^{\prime}}    \ar@{.>}[r]^{(\omega'_{\chi})_G}
&  \mathrm{Ind}_U^G(\psi)_{\lambda^{\prime}}  \\
  {\mathcal F}_{B'}(\chi)^{\mathcal{I}^{\prime}}   \ar[u]^{\Phi_{\chi}} \ar[r]^{\omega^{\prime}_{\chi}} &\mathrm{Ind}_{U^{\prime}}^{G^{\prime}}(\psi)^{\mathcal{I}^{\prime}}. \ar[u]_{\Xi}   \\
      }
\end{equation}

On the other hand, let $\omega_{\chi}:  \mathcal{F}_P(\sigma \otimes \chi) \rightarrow \mathrm{Ind}_U^G(\psi)$ be the Whittaker map attached to the Whittaker functional $\Omega_{\chi}$ through Frobenius reciprocity. Let 
\[
(\omega_{\chi})_{\ast}:  \mathcal{F}_P(\sigma \otimes \chi)_{\lambda'} \rightarrow \mathrm{Ind}_U^G(\psi)_{\lambda'}
\]
be the corresponding map of ${\mathcal H}$-modules. Since 
\[
\mathrm{Hom}_{G}( \mathcal{F}_P(\sigma \otimes \chi),\mathrm{Ind}_U^G(\psi)) \cong \mathrm{Hom}_{\mathcal{H}}( \mathcal{F}_P(\sigma \otimes \chi)_{\lambda'}, \mathrm{Ind}_U^G(\psi)_{\lambda'})
\]
is one dimensional (cf. \cite{cass2}), there is a scalar $a_{\chi}$ so that 
\begin{equation}\label{achi}
(\omega_{\chi})_{\ast}=a_{\chi}(\omega'_{\chi})_G.
\end{equation}

Specializing to $\chi=\chi_s$ (see \S\ref{LS}), we obtain $(\omega_s)_{\ast}:= (\omega_{\chi_s})_{\ast}$, $(\omega'_{s})_G:=(\omega'_{\chi_s})_G$ and $a_s:=a_{\chi_s}$ as functions on a rank-one complex torus whose ring of regular functions is ${\mathbb C}[q^{s/2},q^{-s/2}]$. 
\begin{proposition}\label{GL(n)-Whittaker}
Keeping the above notation, $a_s$ is a monomial in $q^{s/2}$. There is a $\alpha\in {\mathbb C}^{\times}$ and $f\in {\mathbb Z}$ so that for $\chi_s$ with ${\iota}^{G'}_{B'}(\chi_s)$ irreducible ($\Leftrightarrow s\neq \pm 1$), we have 
\[
(\omega_s)_{\ast}(\phi_{K',\chi_s})=\alpha q^{-(fs)/2}(1-q_{\a}^{-s-1})\Xi({\mathcal W}'_{s,\rm sp}).
\]
Here, ${\mathcal W}_{s,\rm sp}'$ is the normalized Whittaker function on $G'$ (w.r.t. $\psi$) satisfying  ${\mathcal W}_{s,\rm sp}'(\begin{smallmatrix}\varpi_{\k}^{-1}&\\&1\end{smallmatrix})=1$. 
\end{proposition} 
\begin{proof}
First, the functions $s\mapsto \omega'_s$ and $s\mapsto \omega_s$ are regular in a certain sense and belong to the polynomial ring ${\mathbb C}[q^{s/2},q^{-s/2}]$ (cf. \cite[Lemma 2.2]{Sha10}). This implies the same for $(\omega_s)_{\ast}$ and $(\omega'_s)_G$. (Note that $\Phi_{\chi_s}$ as a function of $\chi_s$ is regular.) Consequently $s\mapsto a_s$ is a polynomial in $q^{s/2},q^{-s/2}$. On the other hand since both $(\omega_s)_{\ast}$ and $(\omega'_s)_G$ are non-zero, it follows that $a_s$ must be a monomial. Write $a_s=\alpha q^{-(fs)/2}$, $\alpha\in {\mathbb C}^{\times}, f\in {\mathbb Z}$. Using Lemma~\ref{GL(2)-local-coefficient} we see that 
\[
\omega'_{s}(\phi'_{K',\chi_s})=(1-q_{\a}^{-s-1}){\mathcal W}'_{s,\rm sp}
\]
for $s\neq \pm 1$. Evaluating equation (\ref{achi}) at $\phi'_{K',\chi_s}$ and chasing diagram (\ref{local-coefficient-commute}) we get the desired conclusion.
\end{proof}

This proposition immediately implies the following corollary which is also the content of \cite[Corollary 4.1.7]{Kim}. However, the result in loc.cit. is incomplete due to undetermined ``sign" factors and is also missing the crucial ``monomial term" which we believe should be present. We fix that here and note in passing that unlike the proof in loc.cit. our proof does not rely on the theory of intertwining operators. 

\begin{corollary}[The functional equation for Whittaker functions] Suppose $s\neq \pm 1$. Then
\[
 (\omega_{-s})_{\ast}(\phi_{K',\chi_{-s}})=q^{fs}\frac{L(s+1,\pi \times \check{\pi})}{L(-s+1,\pi \times \check{\pi})} (\omega_s)_{\ast}(\phi_{K',\chi_{s}}).
 \]
\end{corollary}

\begin{proof}
One simply has to observe that ${\sf w}_0 (\chi_s)=\chi_{-s}$ and \eqref{L-equal-RSLS}. We apply Proposition~\ref{GL(n)-Whittaker} twice.  
\end{proof}

We now present the precise shape of the local coefficient $C_{\psi}(s,\pi\times\pi)$. 
\begin{theorem}
\label{thm-nonsplit-lc}
Let $\pi\cong \mathrm{c}\text{-}\mathrm{Ind}^{GL_n(F)}_{E^{\times}J(\beta,\mathfrak{A})}(\widetilde{\lambda})$ be an irreducible supercuspidal representation of $GL_n(F)$. Then
\begin{equation}
\label{nonsplit-lc}
  C_{\psi}(s,\pi \times \pi)=\mathrm{vol}(J^{\prime} \cap N)^{-1}\omega_{\pi}(-1)v_{{\sf w}_0}^{-1}q^{-fs} \frac{L(1-s,\check{\pi} \times \pi )}{L(s,\pi \times \check{\pi})}
\end{equation}
with $f \in \mathbb{Z}$ as in Proposition~\ref{GL(n)-Whittaker}.
\end{theorem}
\begin{proof}
Since $  C_{\psi}(s,\pi \times \pi)$ is a rational function of $q^{-s}$, it is enough to prove the assertion on a Zariski open subset of $X(L)$. In particular, we may assume the relevant induced representations are all irreducible. We apply $(\omega_{-s})_{\ast}$ to both sides of the equation in Proposition \ref{GL(n)-intertwining} and utilize the above corollary to get 
\begin{equation}
\label{dual-nonsplit}
\begin{split}
  \omega_{-s} \circ A(\chi_s,\sigma,{\sf w}_0)\circ \phi_{K^{\prime},\chi_s}
  &=\mathrm{vol}( J^{\prime} \cap N) v_{{\sf w}_0}\omega_{\pi}(-1) \frac{L(s,\pi \times \check{\pi})}{L(s+1,\pi \times \check{\pi})}(\omega_{-s} \circ \phi_{K^{\prime},\chi_{-s}})\\
  &=\mathrm{vol}( J^{\prime} \cap N) v_{{\sf w}_0}\omega_{\pi}(-1)q^{fs}\frac{L(s,\pi \times \check{\pi})}{L(1-s,\pi \times \check{\pi})} (\omega_s\circ \phi_{K',\chi_s})
\end{split}
\end{equation}
Now, appealing to the definition of $C_{\psi}(s,\pi\times\pi)$ (cf. \eqref{lc}), we obtain
\[
  C^{-1}_{\psi}(s,\pi \times \pi) (\omega_s)_{\ast}(\phi_{K',\chi_s}) =\mathrm{vol}( J^{\prime} \cap N) v_{{\sf w}_0}\omega_{\pi}(-1)q^{fs}\frac{L(s,\pi \times \check{\pi})}{L(1-s,\pi \times \check{\pi})} (\omega_s)_{\ast}( \phi_{K',\chi_s}).
\]
The result now follows since $(\omega_s)_{\ast}(\phi_{K',\chi_s})\neq 0$. 
\end{proof}

Put 
\[
 \varepsilon^{\rm LS}(s,\pi \times {\pi},\psi)=\mathrm{vol}(J^{\prime} \cap N)^{-1}\omega_{\pi}(-1)v_{{\sf w}_0}^{-1}q^{-fs}.
\]
The integer can be viewed as the ``Langlands-Shahidi conductor" and we denote it as $f^{\rm LS}(\pi\times\pi,\psi)$. It is not difficult to check that its dependence on $\psi$ is given by 
\[
f^{\rm LS}(\pi\times\pi,\psi)=-n^2\ell_{\psi}+f^{\rm LS}(\pi\times\pi),
\]
where $f^{\rm LS}(\pi\times\pi)$ is independent of $\psi$ (cf. \cite[Theorem 2.1 (iv)]{HL}). Of course, for us, $\ell_{\psi}=1$ by choice. In sum \eqref{nonsplit-lc} takes the form 
\[
  C_{\psi}(s,\pi \times \pi)=\varepsilon^{\rm LS}(s,\pi \times \pi, \psi) \frac{L(1-s,\pi\times\check{\pi})}{L(s,\pi \times \check{\pi})}.
\]

\subsection{ Re-visiting the Plancherel constant}
\label{revisit}
For any $\pi_1,\pi_2\in {\mathcal A}_{n}^{0}(F)$, the \textit{Plancherel constant} is a scalar valued function $\mu(s,\pi_1\times\pi_2)\in {\mathbb C}$ attached to the pair $(\pi_1,\pi_2)$  by the defining relation 
\[
   A(-s,\pi_2 \times \pi_1) \circ  A(s,\pi_1 \times \pi_2)
=\mu(s,\pi_1 \times \pi_2)^{-1}
\]
on a Zariski open subset of ${\mathbb C}$. It is a rational function in $q^{-s}$ and clearly depends on the measures defining intertwining operators. By \cite[Proposition 3.11]{Sh81}, we have 
\begin{equation}\label{plancherel-lc}
\mu(s,\pi_1\times\pi_2)=C_{\psi}(s,\pi_1\times\pi_2)C_{\psi}(-s,\pi_2\times\pi_1).
\end{equation}
Now, we return to $\pi_1 \cong \mathrm{c}\text{-}\mathrm{Ind}^{GL_n(F)}_{E^{\times}J(\beta,\a)}(\tilde{\lambda}_1)$ and $\pi_2 \cong  \mathrm{c}\text{-}\mathrm{Ind}^{GL_n(F)}_{E^{\times}J(\beta,\a)}(\tilde{\lambda}_2)$ as in \S\ref{LS} so that they are unitary and associated to the same endo-class. In the two cases considered there, we have the following expressions for the plancherel constant:
\begin{itemize}
\item Suppose $\pi_2 \not\cong \pi_1 \otimes(\chi\circ\det)$ for any unramified character $\chi$ of $F^{\times}$. It follows from Proposition~\ref{lemma-1} that 
\[
\mu(s,\pi_1\times\pi_2)\equiv {\rm vol}(J\cap N)^{-1}{\rm vol}(J\cap\overline{N})^{-1}.
\]
\item Suppose $\pi=\pi_1\cong \pi_2$. Applying Proposition \ref{GL(n)-intertwining} with $\chi=\chi_s$, we obtain 
\[
\label{volume-plancherel}
 \mu(s,\pi \times \pi)
 =\mathrm{vol}(J^{\prime} \cap N)^{-2}v_{{\sf w}_0}^{-2} \frac{L(1+s,\pi \times \check{\pi})}{L(s,\pi \times \check{\pi})} \frac{L(1-s, \pi\times \check{\pi} )}{L(-s,\pi\times \check{\pi} )}.
\]
\end{itemize}

We begin with certain volume computations. Choose Haar measures $dn$ on $N$ and $d\overline{n}$ on $\overline{N}$ satisfying the condition in \S\ref{sec-measures} relative to our chosen $\psi$. In short, we refer to $(dn,d\overline{n})$ as a dual pair of measures. Let us normalize the measure on $L=GL_n(F)\times GL_n(F)$ so that
\[
 \mathrm{vol}(H^1(\beta,\mathfrak{a}))=\frac{\mathrm{vol}_F(e_n \mathfrak{p}^{m+1})\mathrm{vol}(J\cap N)\mathrm{vol}(J\cap\overline{N})}{[(U_n(F)\cap J(\beta,\a)): (U_n(F)\cap H^1(\beta,\a))]}.
  \]
This amounts to saying that the measure on $L$ is so that 
\begin{equation}
\label{L-normal}
 \upsilon \mathrm{vol}(J\cap N)\mathrm{vol}(J\cap\overline{N})=1,
\end{equation}
where $v$ is as in Proposition~\ref{RS-gamma}. Note that this is independent of the choice of the dual pair $(dn,d\overline{n})$. In fact, we have the following:
\begin{lemma}
\label{split-volume}
For a dual pair of Haar measures $(dn,d\overline{n})$, we have
\[
 \mathrm{vol}(J^{\prime}\cap N) \mathrm{vol}(J^{\prime}\cap\overline{N})=\mathrm{vol}(J\cap N)\mathrm{vol}(J\cap\overline{N})=q_{\mathfrak{a}}^m.
\]
\end{lemma}
\begin{proof}
It is clear that the product of volumes $\mathrm{vol}(J\cap N) \mathrm{vol}(J\cap\overline{N})$ is invariant under conjugation by $\begin{psmallmatrix}\varpi_{E}^{m+1}I_n&\\&I_n\end{psmallmatrix}$. Therefore
\[
 \mathrm{vol}(J\cap N) \mathrm{vol}(J\cap\overline{N})=\mathrm{vol}(J^{\prime}\cap N) \mathrm{vol}(J^{\prime}\cap\overline{N}).
\]
Since the cover $(J',\lambda')$ is the same in both cases, we may assume that we are in the split case, i.e., the first bullet above, for the purpose of calculating the said product of volumes. In this situation, with the measure on $L$ normalized as above, it follows from (\ref{L-normal}) and Theorem~\ref{LS-equal} that 
\[
C_{\psi}(s,\pi_1\times\pi_2)=\omega_{\pi_2}(-1)^n\gamma(s,\pi_1\times\check{\pi}_2,\psi)
\]

Appealing to \eqref{sum-shell}, we have 
\[
\begin{split}
&C_{\psi}(-s,\pi_2 \times \pi_1) C_{\psi}(s,\pi_1 \times \pi_2)\\
&=\omega_{\pi_1}(-1)\omega_{\pi_2}(-1) \mathrm{vol}(J\cap N)^{-2} \mathrm{vol}(J\cap\overline{N})^{-2} \\
&\quad \times \int\limits_{J(\beta,\a)}  {\mathcal W}_1((\varpi_E^{m}X)^{-1}){\mathcal W}_2(\varpi_E^mX) \phi_m(X) dX
\int\limits_{J(\beta,\a)}  {\mathcal W}_1(\varpi_E^mX){\mathcal W}_2((\varpi_E^{m}X)^{-1}) \phi_m(X) dX\\
&=\mathrm{vol}(J\cap N)^{-2} \mathrm{vol}(J\cap\overline{N})^{-2} |I_m|^{2}.\\
\end{split}
\]
Here, the second equality follows from Proposition \ref{Whitt} along with changing the variable $-X \mapsto X$ in one of the inetgrals. Combining (\ref{plancherel-lc}) with the corresponding expression for the plancherel constant, we get 
\[
|I_m|^2=\mathrm{vol}(J\cap N) \mathrm{vol}(J\cap\overline{N}).
\]
On the other hand, using the functional equation \eqref{RS-FE} and Proposition \ref{RS-gamma}, we obtain 
\[
 1=\gamma(1-s,\check{\pi}_1 \times \pi_2,\overline{\psi})\gamma(s,\pi_1 \times \check{\pi}_2,\psi )
 =\upsilon^2 q_{\mathfrak{a}}^m |I_m|^{2}.
\]  
Taking the normalization \eqref{L-normal} into account, we conclude 
\[
\mathrm{vol}(J\cap N) \mathrm{vol}(J\cap\overline{N}) =\upsilon^{-1}=q_{\a}^{m}. 
\]
\end{proof}
\begin{remark}
A different computation \cite[Theorem 6.5]{BHK} shows that the above product of volumes is also given as 
\[
 \mathrm{vol}(J^{\prime}\cap N) \mathrm{vol}(J^{\prime}\cap\overline{N})=q^{-\frac{n^2\mathfrak{c}(\beta)}{d^2}},
\]
where $\mathfrak{c}(\beta)$ is ``the generalized discriminant" as in \cite[\S6.4]{BHK}. Equating the two expressions for the product of volumes gives the relation 
\[
 \frac{n\mathfrak{c}(\beta)}{d^2} =-\frac{m}{e}.
\]
This seems to generalize the result \cite[\S5A]{Tam} which establishes this equality only for tame supercuspidal representations. 
\end{remark}

\begin{lemma}\label{volume-conductor}
With the measure on the Levi subgroup $L$ fixed as above and taking the measure $dn$ to be the one that underlies Proposition~\ref{GL(n)-intertwining}, for the corresponding dual pair $(dn,d\overline{n})$ we have 
\[
{\rm vol}(J'\cap N)=q_{\a}^{\frac{m+1}{2}}v_{{\sf w}_{0}}^{-1};\quad {\rm vol}(J'\cap \overline{N})=q_{\a}^{\frac{m-1}{2}}v_{{\sf w}_{0}}.
\]
\end{lemma}

\begin{proof}
We follow the notation of Proposition \ref{GL(n)-intertwining}. Apply both sides of equation \eqref{eqn-Aint} to $\phi_{1,\chi}$ and use Lemma~\ref{intertwining-spherical} to obtain 
\[
 A(\chi,\sigma,{\sf w}_0)\circ \phi_{1,\chi}=c(\chi,\sigma)((c_{{\sf w}_0}(\chi)-1)\phi_{1,\chi}+q_{\a}^{-1}\omega_{\pi}(-1)v_{{\sf w}_0}\phi_{{\sf w}_0,\chi}).
\]
Evaluating this equation at a ``symmetric" $w\in W$ and then evaluating the resulting equation at ${\sf w}_{0}$, we get 
\[
\begin{split}
 { \mathrm{vol}(J')^{-1}} q^{-1}_{\mathfrak{a}}\omega_{\pi}(-1)v_{{\sf w}_0}c(\chi,\sigma)\varphi_{w}
 &=A(\chi,\sigma,{\sf w}_0)(\phi_{1,\chi}(w))({\sf w}_{0})\\
 &=\int\limits_N f_{1,w,\chi}({\sf w}_0^{-1}n{\sf w}_0) dn
 =\int\limits_{J^{\prime} \cap \overline{N}}f_{1,w,\chi}(\overline{n}) d\overline{n}\\
& ={ \mathrm{vol}(J')^{-1}} {\rm vol}(J^{\prime} \cap \overline{N})\varphi_w.
\end{split}
\]
Substituting for $c(\chi,\sigma)$, we conclude from \eqref{c-value} that
\[
{\rm vol}(J^{\prime} \cap \overline{N})=q_{\mathfrak{a}}^{-1}v^2_{{\sf w}_0}{\rm vol} (J^{\prime}\cap N).
\]
Then Lemma \ref{split-volume} implies that 
\[
v^2_{{\sf w}_0}{\rm vol}(J^{\prime} \cap {N})^2=q_{\a}^{m+1}
\]
and the desired conclusion follows from this. 
\end{proof}

Lemma~\ref{split-volume} and Lemma~\ref{volume-conductor} together with our expression for the Plancherel constant gives another proof of the following well-known result of Shahidi \cite[Theorem 6.1]{Sha6}:

\begin{theorem} 
For a pair $(\pi_1,\pi_2)$ of irreducible unitary supercuspidal representations of $GL_n(F)$ sharing the same endo-class as above, there exists a unique measure $dn\otimes d\overline{n}$ defining the intertwining operators so that \[
\mu(s,\pi_1\times\pi_2)
 =q^{f(\pi_1 \times \check{\pi}_2,\psi)} \frac{L(s,\pi_1\times \check{\pi}_2)}{L(1+s,\pi_1 \times \check{\pi}_2)} \frac{L(-s,\check{\pi}_2 \times \pi_1)}{L(1-s,\check{\pi}_2\times \pi_1)}.
 \]

\end{theorem}

\begin{remark}
In the course of the proof of Lemma~\ref{split-volume}, we proved \eqref{eqn-equal} in the split case. We are unable to deduce this equality in the non-split case directly through local means. It likely involves a better understanding of the map $\Xi$ in Proposition~\ref{GL(n)-Whittaker}. In any case, with the volume factors in place, we may appeal to the functional equation satisfied by the local coefficient \cite{shahidi-plancherel} (and \cite{HL} in the case of positive characteristic) whose proof invokes a global-to-local argument, and show $f^{\rm LS}(\pi\times\pi,\psi)=f(\pi\times\check{\pi},\psi)$. This in turn implies \eqref{eqn-equal} in the non-split case. In general, it is difficult to prove such equalities using purely local methods. We view Shahidi's work \cite{Sha6} as a miraculous calculation, but his methods do not seem to generalize to other situations. In contrast, the approach here is more `formal' and can be adapted to other situations as pointed out in the Introduction. 
\end{remark}

\begin{acknowledgments} The first-named author would like to thank R. Ye and E. Zelingher for a very useful discussion about the choice of the additive character. The second-named author would like to thank Phil Kutzko for helpful discussions about this paper. We thank an anonymous referee for his/her careful reading of our paper and suggestions on the manuscript. 
\end{acknowledgments}

 \bibliographystyle{amsplain}

\end{document}